\documentclass[12pt]{article}
\usepackage{amsthm,amsfonts,amsmath,amscd,amssymb,epsfig}
\usepackage[left=2.8cm, right=2cm]{geometry}
\usepackage{epigraph} 
\numberwithin{equation}{section}
\theoremstyle{plain}
\newtheorem{theorem}{Theorem}[section]

\newtheorem{proposition}[theorem]{Proposition}
\newtheorem{corollary}[theorem]{Corollary}
\newtheorem{lemma}[theorem]{Lemma}
\newtheorem{lemma-definition}[theorem]{Lemma-Definition}

\theoremstyle{definition}
\newtheorem{definition}[theorem]{Definition}

\newtheorem{remark}[theorem]{Remark}

\newtheorem{example}[theorem]{Example}

\newcommand{\Addresses}{{
  \bigskip
\footnotesize
\textsc{Universit\'e de Strasbourg\\
Institut de recherche math\'ematique avanc\'ee (IRMA), \\Strasbourg, France}\par\nopagebreak
\textit{E-mail address:} \texttt{shahmath19@gmail.com, shah.faisal@unistra.fr}}}
\graphicspath{{Images/}}
\begin{document}
\title{A proof of Gromov's non-squeezing theorem}
\author{Shah Faisal}
\date{}
\maketitle
\begin{abstract}
The original proof of the Gromov's non-squeezing theorem \cite{Gromov85} is based on pseudo-holomorphic curves. The central ingredient is the compactness of the moduli space of pseudo-holomorphic spheres in the symplectic manifold $(\mathbb{CP}^1\times T^{2n-2}, \omega_{\mathrm{FS}}\oplus \omega_{\mathrm{std}})$ representing the homology class $[\mathbb{CP}^1\times\{\operatorname{pt}\}]$.  In this article, we give two proofs of this compactness. The fact that the moduli space carries the minimal positive symplectic area is essential to our proofs. The main idea is to reparametrize the curves to distribute the symplectic area evenly and then apply either the mean value inequality for pseudo-holomorphic curves or the Gromov-Schwarz lemma to obtain a uniform bound on the gradient. Our arguments avoid bubbling analysis and Gromov's removable singularity theorem, which makes our proof of Gromov's non-squeezing theorem more elementary.
 \end{abstract}	
\tableofcontents
\parskip=4pt
\section{Introduction}
Let $(x_1, y_1,\dots,x_n, y_n)$ denote the standard coordinates on the Euclidean space $\mathbb{R}^{2n}$.
The standard open ball of capacity $r>0$, denoted by $B^{2n}(r)$, is defined by
\[B^{2n}(r):=\big\{(x_1,y_1,\dots, x_n,y_n)\in \mathbb{R}^{2n}: \sum_{i=1}^n\pi(x_i^2+y_i^2) < r\big\}.\]
We equip $B^{2n}(r)$ with the standard symplectic form $\omega_{\mathrm{std}}:=\sum_{1}^{n}dx_i\wedge dy_i$.

The celebrated Gromov's non-squeezing theorem is stated as follows.
\begin{theorem}[\cite{Gromov85}]\label{gromov-nonsqueezing}
	There exists a symplectic embedding 
	\[\psi: (B^{2n}(r), \omega_{\mathrm{std}})\to (B^2(R)\times \mathbb{R}^{2n-2}, \omega_{\mathrm{std}})\] 
if and only if $r\leq R$.
\end{theorem}
The ``if'' part of this theorem is trivial: for $r\leq R$, the inclusion is a symplectic embedding. The following more general theorem implies the ``only if'' part (cf. Corollary \ref{nons1}). 
\begin{theorem}\label{nons}
Let $(M, \omega)$ be a closed symplectic manifold of dimension $(2n-2)\geq 2$ with vanishing second homotopy group, i.e., $\pi_2 (M)=0$. Let $\sigma $ be an area form on $\mathbb{CP}^1$. If there exists a symplectic embedding 
\[\psi: ( B^{2n}(r), \omega_{\mathrm{std}})\to (\mathbb{CP}^1\times M, \sigma\oplus \omega),\] 
then 
	\[r\leq \int_{\mathbb{CP}^1}\sigma.\]
\end{theorem}

The proof of this theorem is based on pseudo-holomorphic curves theory. To be more specific, the following existence result plays the main role in the proof. \begin{theorem}\label{ab} Assume the setup of Theorem \ref{nons}. Given any $(\sigma\oplus \omega)$-compatible almost complex structure $J$ on $\mathbb{CP}^1\times M$, for every point $p\in \mathbb{CP}^1\times M$ there exists a $J$-holomorphic sphere $u:(\mathbb{CP}^1,i)\to (\mathbb{CP}^1\times M,J)$ that passes through $p$ and represents the homology class $[\mathbb{CP}^1\times \{\operatorname{pt}\}] \in H_2(\mathbb{CP}^1\times M,\mathbb{Z} )$.
\end{theorem}

Given Theorem \ref{ab}, let us prove Theorem \ref{nons}.
\begin{proof}[Proof of Theorem \ref{nons}]
	Suppose there exists a symplectic embedding
	\[\psi: ( B^{2n}(r), \omega_{\mathrm{std}})\to (\mathbb{CP}^1\times M, \sigma\oplus \omega).\] 
For each $\epsilon\in (0,r)$,  $\psi$ restricts to a symplectic embedding of the closed ball
	\[\psi: ( \bar{B}^{2n}(r-\epsilon), \omega_{\mathrm{std}})\to (\mathbb{CP}^1\times M, \sigma\oplus \omega).\] 

By Proposition \ref{exist}, choose an $(\sigma\oplus \omega)$-compatible almost complex structure $J_\epsilon$ on $\mathbb{CP}^1\times M$ that agrees with $\psi_{*}J_{\mathrm{std}}$ on $\psi(  \bar{B}^{2n}(r-\epsilon))$, where $\psi_{*}J_{\mathrm{std}}$ is the push forward of the standard complex structure $J_{\mathrm{std}}$ on $ \bar{B}^{2n}(r-\epsilon)$. By Theorem \ref{ab}, there exists a $J_\epsilon$-holomorphic sphere $u_\epsilon :(\mathbb{CP}^1,i)\to \mathbb{CP}^1\times M$ in the homology class $[\mathbb{CP}^1\times \{\operatorname{pt}\}]$ passing through $\psi(0)$. Note that 
\begin{equation}\label{estimate}
\int_{\mathbb{CP}^1} u_\epsilon^{*}(\sigma\oplus \omega)=\langle [u_\epsilon],\sigma\oplus \omega\rangle=\langle [\mathbb{CP}^1\times \{\operatorname{pt}\}],\sigma\rangle=\int_{\mathbb{CP}^1}\sigma.
\end{equation}

The image of $u_\epsilon$ is not contained in $\psi(  B^{2n}(r))$: if it is, then 
by Stokes' theorem we have
\[\int_{\mathbb{CP}^1} u_\epsilon^{*}(\sigma\oplus \omega)=\int_{\mathbb{CP}^1} u_\epsilon^{*}\psi_{*} \omega_{\mathrm{std}}=\int_{\mathbb{CP}^1} u_\epsilon^{*}\psi_{*} (d\lambda_{\mathrm{std}})=\int_{\mathbb{CP}^1} d(u_\epsilon^{*}\psi_{*} \lambda_{\mathrm{std}})=0.\]
By Corollary \ref{comparison0}, $u_\epsilon$ is constant. This is a contradiction to $[u_\epsilon]=[\mathbb{CP}^1\times \{\operatorname{pt}\}]$. Thus,
\[\psi^{-1}\circ u_\epsilon: u_\epsilon^{-1}(\psi( \bar{B}^{2n}(r-\epsilon)))\to  \bar{B}^{2n}(r-\epsilon)\]
 is a  $J_{\mathrm{std}}$-holomorphic curve with boundary mapping to $\partial \bar{B}^{2n}(r-\epsilon)$ and passing through the center of $\bar{B}^{2n}(r-\epsilon)$. 
  By Lemma \ref{monoi}, we have
	\[r-\epsilon\leq \int_{u_\epsilon^{-1}(\psi( \bar{B}^{2n}(r-\epsilon)))}(\psi^{-1}\circ u_\epsilon)^*\omega_{\mathrm{std}}=\int_{u_\epsilon^{-1}(\psi( \bar{B}^{2n}(r-\epsilon)))}u_\epsilon^*(\sigma\oplus \omega)\leq \int_{\mathbb{CP}^1}u_\epsilon^*(\sigma\oplus \omega).\]
From (\ref{estimate}), it follows that
\[r-\epsilon\leq \int_{\mathbb{CP}^1}\sigma.\]
Since $\epsilon>0$ is arbitrary, we have $r\leq \int_{\mathbb{CP}^1}\sigma$.
\end{proof}
The  proof above uses the existence of
a pseudo-holomorphic curve to give the symplectic embedding obstruction $r\leq \int_{\mathbb{CP}^1}\sigma$.  Pseudo-holomorphic curves are currently the most important tool for dealing with symplectic embedding problems. A principle of Eliashberg  {\cite[p. 169]{Schlenk2018}} states that a pseudo-holomorphic curve can describe any obstruction to a symplectic embedding.

\begin{corollary}[cf. {\cite[Theorem 1]{Abbondandolo2015}}]\label{nons1} 
Let $S$ be any symplectic $2$-plane in $(\mathbb{R}^{2n}, \omega_{\mathrm{std}}:=\sum_{1}^{n}dx_i\wedge dy_i)$, i.e., a $2$-plane on which $\omega_{\mathrm{std}}$ does not vanish. Let $\pi_S:\mathbb{R}^{2n}\to S$ be the projection along the symplectic orthogonal complement $S^{\bot \omega_{\mathrm{std}}}$ of $S$. Let $\operatorname{Area}_{\omega_{\mathrm{std}}}$ denote the area on $S$ induced by ${\omega_{\mathrm{std}}}$. For any symplectic embedding
	 \[\psi: (B^{2n}(r), \omega_{\mathrm{std}})\to (\mathbb{R}^{2n},\omega_{\mathrm{std}})\]
we have
	\[\operatorname{Area}_{\omega_{\mathrm{std}}}(\pi_S(\psi(B^{2n}(r)))) \geq r,\] i.e.,  the shadow of any symplectic image of the ball $ B^{2n}(r)$ on any symplectic plane in $\mathbb{R}^{2n}$ is at least as large as the shadow of $ B^{2n}(r)$.
\end{corollary}
Corollary \ref{nons1} is equivalent to Theorem \ref{gromov-nonsqueezing}. This will become apparent after the next two proofs.
\begin{proof}[Proof of Theorem \ref{gromov-nonsqueezing}]
We show that if there exists a symplectic embedding 
	\[\psi: ( B^{2n}(r), \omega_{\mathrm{std}})\to ( B^2(R)\times \mathbb{R}^{2n-2},\omega_{\mathrm{std}}),\]
then $r\leq R$. 

Suppose such an embedding exists. For each $\epsilon\in (0,r)$, this  embedding restricts to a symplectic embedding
	\[\psi: ( \bar{B}^{2n}(r-\epsilon), \omega_{\mathrm{std}})\to ( B^2(R)\times \mathbb{R}^{2n-2},\omega_{\mathrm{std}}).\]
	The image $\psi(\bar{B}^{2n}(r-\epsilon))$ is compact. Choose $l>0$ large so that $ B^2(R)\times [-l,l]^{2n-2}$ contains $\psi(\bar{B}^{2n}(r-\epsilon))$ in its interior. Since $\omega_{\mathrm{std}}$ is translation invariant, it descends to a symplectic form on the quotient $T^{2n-2}:=\mathbb{R}^{2n-2}/ 2l\mathbb{Z}^{2n-2}$ through the canonical projection $\pi:\mathbb{R}^{2n-2}\to T^{2n-2}$. Therefore, we get a symplectic embedding 
	\[  \bar{B}^{2n}(r-\epsilon)\xrightarrow[]{\psi}  B^2(R)\times \mathbb{R}^{2n-2} \xrightarrow[]{\text{Id}\times \pi}  B^2(R)\times T^{2n-2}.\]
	
	Give $\mathbb{CP}^1$ an area form $\sigma$ of total area $(R+\epsilon)$ and embed $B^2(R)$ into $\mathbb{CP}^1$ symplectically. Such an embedding exists because volume-preserving and symplectic embeddings are the same in dimension $2$. 
	Finally, we get a symplectic embedding 
	\[( B^{2n}(r), \omega_{\mathrm{std}})\to (\mathbb{CP}^1\times T^{2n-2},  \sigma \oplus \omega_{\mathrm{std}}).\]
	Since $\pi_2(T^{2n-2})=0$, Theorem \ref{nons} implies 
	\[r-\epsilon\leq \int_{\mathbb{CP}^1}\sigma=R+\epsilon.\]
Since $\epsilon$ is arbitrary, $\epsilon\to 0$ implies 
\[r\leq R.\qedhere\] 
\end{proof}
\begin{proof}[Proof of Corollary \ref{nons1}]	
On the contrary, suppose there exists a symplectic embedding
	 \[\psi: (B^{2n}(r), \omega_{\mathrm{std}})\to (\mathbb{R}^{2n},\omega_{\mathrm{std}})\]
such that 
\[\operatorname{Area}_{\omega_{\mathrm{std}}}(\pi_S(\psi(B^{2n}(r))))< r.\]
One can map $\pi_S(\psi(B^{2n}(r)))$ to a subset of a ball of capacity $R<r$ in $S$ by an area-preserving diffeomorphism $\phi_S$. The symplectomorphism $\phi_S\times \operatorname{Id}_{S^{\bot \omega_{\mathrm{std}}}}\circ \psi $ maps $B^{2n}(r)$ into $B^2(R)\times \mathbb{R}^{2n-2}$ with $R<r$. This is a contradiction to Theorem \ref{gromov-nonsqueezing}.
\end{proof}
It is clear from above that Theorem $\ref{ab}$ plays a central role in Gromov's non-squeezing theorem. To prove it, we start with an $(\delta\oplus \omega)$-compatible almost complex structure $J_0$ on $\mathbb{CP}^1\times M$ for which we can explicitly write down all $J_0$-holomorphic spheres representing the homology class $[\mathbb{CP}^1\times \{\operatorname{pt}\}]$ and passing through $p$. We show that the count of $J_0$-holomorphic spheres representing $[\mathbb{CP}^1\times \{\operatorname{pt}\}]$ and passing through $p$ is non-zero (cf. Lemma \ref{split0}). Then, for any $\omega_{\mathrm{FS}}\oplus \omega$-compatible almost complex structure $J$ on $\mathbb{CP}^1\times M$, we construct a sequence of almost complex structures $J_k$ that converges to $J$ such that for each $k$, a $J_k$-holomorphic sphere representing $[\mathbb{CP}^1\times \{\operatorname{pt}\}]$ and passing through $p$ exists (cf. Lemma \ref{split2}). The existence for the given $J$ then follows as a consequence of the compactness (cf. Theorem \ref{compact1}) of the following moduli space.
\begin{definition}
Let $(M, \omega)$ be a closed symplectic manifold of dimension $(2n-2)\geq 2$ with vanishing second homotopy group, i.e., $\pi_2 (M)=0$. Let $\omega_{\mathrm{FS}}$ denote the Fubini-Study form on $\mathbb{CP}^1$. Let $\{J_t\}_{t\in [0,1]}\subset \mathcal{J}_c(\mathbb{CP}^1\times M,\omega_{\mathrm{FS}}\oplus \omega)$ be a continous path of $(\omega_{\mathrm{FS}}\oplus \omega)$-compatible almost complex structures.  We define
\begin{equation}\label{moduli}
	\mathcal{M}(\{J_t\}_{t\in [0,1]},[\mathbb{CP}^1\times\{\operatorname{pt}\}]):=\left\{(t,u):
	\begin{array}{l}
		t\in [0,1],\\
		u:(\mathbb{CP}^1,i)\to (\mathbb{CP}^1\times M,J_t),\\
		du\circ i=J_t\circ du ,\\
		u_*[\mathbb{CP}^1]=[\mathbb{CP}^1\times\{\operatorname{pt}\}].
	\end{array}
	\right\}\bigg/\sim
\end{equation}
where $u_1 \sim u_2$ if and only if $u_1=u_2\circ\varphi$ for some $\varphi \in \operatorname{Aut}(\mathbb{CP}^1,i)$.
\end{definition}
\begin{theorem}[cf. {\cite[Theorem 2.4]{Beckschulte_2021}}]\label{compact1}
The moduli space defined by (\ref{moduli}) is compact in the quotient topology coming from  $[0,1]\times C^\infty(\mathbb{CP}^1,\mathbb{CP}^1\times M)$.
\end{theorem} 
Section \ref{outline} below outlines our proof of Theorem \ref{compact1}. A detailed proof is given in Section \ref{mainproof}.

\subsection{Outline of the proof of Theorem \ref{compact1} via mean value inequality}\label{outline}
We briefly explain our proof of Theorem \ref{compact1} that is based on the mean value inequality for pseudo-holomorphic curves described in Theorem \ref{meanvaluek}. Let $g$ be a Riemannian metric on $\mathbb{CP}^1\times M$ and $J$ be an  $(\omega_{\mathrm{FS}}\oplus\omega)$-compatible almost complex structure on $\mathbb{CP}^1\times M$. Consider the moduli space
	\begin{equation*}
	\mathcal{M}(J,[\mathbb{CP}^1\times\{\operatorname{pt}\}]):=\left\{
	\begin{array}{l}
		u:(\mathbb{CP}^1,i)\to (\mathbb{CP}^1\times M,J),\\
		du\circ i=J\circ du ,\\
		u_*[\mathbb{CP}^1]=[\mathbb{CP}^1\times\{\operatorname{pt}\}]\in H_2(\mathbb{CP}^1\times M,\mathbb{Z}).
\end{array}
	\right\}\bigg/\sim
\end{equation*}
where $u_1 \sim u_2$ if and only if $u_1=u_2\circ\varphi$ for some $\varphi \in \operatorname{Aut}(\mathbb{CP}^1,i)$. We show that each $[u]\in \mathcal{M}(J,[\mathbb{CP}^1\times\{\operatorname{pt}\}])$ admits a representative $v$ such that
\begin{equation}\label{C1bound}
\|dv(z)\|_{g}\leq C_{J,g},
\end{equation}
for all $z\in \mathbb{CP}^1$, and some constant $C_{J,g}>0$ that only depends on $(g,J)$. Moreover, the constant $C_{J,g}$ is continuous with respect to $J$ and $g$ in the $C^\infty$-topology.

This is enough to conclude Theorem \ref{compact1}. To see this, let $\{J_t\}_{t\in [0,1]}\subset \mathcal{J}_c(\mathbb{CP}^1\times M,\omega_{\mathrm{FS}}\oplus \omega)$ be a continuous path of $(\omega_{\mathrm{FS}}\oplus \omega)$-compatible almost complex structures. For each $t\in [0,1]$, by (\ref{C1bound}), there exists $C_{J_t,g}>0$ such that  every $[u]\in \mathcal{M}(J_t,[\mathbb{CP}^1\times\{\operatorname{pt}\}])$ admits a representative $v$ such that for all $z\in \mathbb{CP}^1$ we have
\[\|dv(z)\|_{g}\leq C_{J_t,g}.\]
The constant $C_{J_t,g}>0$ only depends on $(g,J_t)$ and varies continuously with $t\in [0,1]$. Since the interval $[0,1]$ is compact, we can choose $C_{J_t,g}$ to be uniform in $t$. 

The topology on the moduli space in Theorem \ref{compact1} is metrizable as a special case of {\cite[Theorem 5.6.6(ii)]{MacDuff}}. So compactness, in this case, is equivalent to sequential compactness. Given a sequence $\{[u_k]\}$ in the moduli space in Theorem \ref{compact1}, there exist a sequence $\{t_k\}$ in $[0,1]$ and a corresponding sequence  $\{J_{t_k}\}$ in $ \{J_t\}_{t\in [0,1]}$ such that $u_k$ is $J_{t_k}$-holomorphic. Since $[0,1]$ is compact, $\{t_k\}$ has a subsequence, still denoted by $\{t_k\}$, that converges to some $t_{\mathrm{lim}} \in [0,1]$. This implies the sequence $\{J_{t_k}\}$ $C^\infty$-converges to $J_{t_{\mathrm{lim}}}\in \{J_t\}_{t\in [0,1]}$ 
because the family $\{J_t\}_{t\in [0,1]}$ is continuous in $C^\infty$-topology. Moreover, $\{u_k\}$ has a uniform $C^0$-bound because the target manifold $\mathbb{CP}^1\times M$ is closed. Also, by the above discussion,  there exists $C>0$ such that (after re-parametrizing  $u_k$) we have
\[\|du_k(z)\|_{g}\leq C,\]
for all $z\in \mathbb{CP}^1$, $k \in \mathbb{Z}_{\geq 1}$. This $C^1$-bound implies a $C^\infty$-bound on the sequence $\{u_k\}$ by {\cite[Sec. 2.2.3]{Casim2014}}.  By Arzel\'a-Ascoli theorem, $u_k$ has a subsequence that $C^\infty$-converges to a $J_{t_{\mathrm{lim}}}$-holomorphic map $u:\mathbb{CP}^1\to \mathbb{CP}^1\times M$. Using $C^0$-convergence, the limit $u$ represents the class $[\mathbb{CP}^1\times\{\operatorname{pt}\}]$.
 Below we outline a proof of ($\ref{C1bound}$). A detailed proof is given in Section \ref{mainproof}.
 \begin{itemize}
\item[\textbf{Step 01}]  For any smooth map $u:\mathbb{CP}^1\to \mathbb{CP}^1\times M$, we have
\[E(u):=\int_{\mathbb{CP}^1} u^{*}(\omega_{\mathrm{FS}}\oplus \omega)=m\pi\]
for some integer $m$ depending on $u$. This means that any smooth map $u:\mathbb{CP}^1\to \mathbb{CP}^1\times M$ with symplectic area less than $\pi$ and greater than $-\pi$ must have zero symplectic area. If $u$ is not constant and is $J$-holomorphic for some $\omega_{\mathrm{FS}}\oplus\omega$-compatible almost complex structure $J$, then  $m>0$ because $E(u)>0$ by Corollary \ref{comparison0}. Moreover, $m=1$ if $u$ represents the class $[\mathbb{CP}^1\times \{\operatorname{pt}\}]$. The conclusion is that $J$-holomorphic spheres in $\mathcal{M}(J,[\mathbb{CP}^1\times\{\operatorname{pt}\}])$ have the minimal positive symplectic area (namely $\pi$) for any $(\omega_{\mathrm{FS}}\oplus \omega)$-compatible almost complex structure $J$.

\item[\textbf{Step 02}] Consider $g_1$, $g_2$ and $g_3$ $\in \operatorname{Aut}(\mathbb{CP}^1,i)$  given by
\[ \left\{
\begin{array}{ll}
	g_1(z)=\lambda_{1} z,&\\
	g_2(z)=\frac{z+\lambda_{2}}{z\lambda_2+1},&\\
	g_3(z)=\frac{z+\lambda_{3}}{-\lambda_{3}z+1},&\\
	
\end{array}
\right. \]
for $\lambda_1, \lambda_2, \lambda_3 \in \mathbb{C}$. For each $u\in \mathcal{M}(J,[\mathbb{CP}^1\times\{\operatorname{pt}\}])$, choose $\lambda_1, \lambda_2$ purely real and $\lambda_3$ purely imaginary such that $v:=u \circ g_1 \circ g_2 \circ g_3$ has the symplectic area distribution
\[ \left\{
\begin{array}{ll}
	E(v|_{\mathbb{D}^2})=\pi/2,&\\
	E(v|_{\operatorname{Re}(z)\geq 0})=\pi/2,&\\
	E(v|_{\operatorname{Imag}(z)\geq 0})=\pi/2,&\\
\end{array}
\right. \]
where $\mathbb{D}^2$ is the unit disk centered at the origin in $\mathbb{C}$ corresponding to the lower hemisphere on $\mathbb{CP}^1$ under the stereographic projection.

\item [\textbf{Step 03}] For $z\in \mathbb{CP}^1$, denote the Fubini-Study disk of radius $\pi/24$ centered at $z$ by $B_\mathrm{FS}(z,\pi/24)$. Let $\operatorname{injrad}(\mathbb{CP}^1\times M,g_0)$ denote the injectivity radius of $\mathbb{CP}^1\times M$ with respect to the Riemannian metric $g_0:=(\omega_{\mathrm{FS}}\oplus \omega)(\cdot,J\cdot)$. There is a constant $k>0$ that depends only on $g_0$ and varies continuously with respect to $g_0$ in $C^\infty$-topology such that the following holds: for any $c\geq \max\{e^{4k \pi}, e^{\frac{18\pi^2}{\operatorname{injrad}(\mathbb{CP}^1\times M,g_0)^2}} \}$ we have 
\begin{equation}\label{cut12}
	\int_{B_{\mathrm{FS}}(z,r_v)}v^{*}(\omega_{\mathrm{FS}}\oplus\omega)\leq k \frac{2\pi^2}{\log(c)}.
\end{equation}
for some $r_v\in (\frac{\pi}{24c},\frac{\pi}{24})$ that depends on the map $v$. Here $c>1$ is arbitrary and does not depend on $v$. To obtain the estimate (\ref{cut12}), we use the fact that $v$ has minimal positive symplectic area, by Step 01,  and has the symplectic area distribution obtained in Step 02 by a suitable rescaling.

\item[\textbf{Step 04}] Let $c_{J,g_0}>0$ be the positive constant in Theorem \ref{meanvaluek}. Choose 
\[c=\max\{e^{4k \pi},e^{\frac{18\pi^2}{\operatorname{injrad}(\mathbb{CP}^1\times M,g_0)^2}} , e^{2k\pi^2c^{-1}_{J,g_0}}\}\]
 in (\ref{cut12}). By Corollary \ref{comparison0}, we have
\[\int_{B_{\mathrm{FS}}(z,r_v)}\|dv\|_{g_0}^2=\int_{B_{\mathrm{FS}}(z,r_v)}v^{*}(\omega_{\mathrm{FS}}\oplus\omega)<c_{J,g_0}.\]
By  Theorem \ref{meanvaluek}, we have 
\[\|dv(z)\|_{g_0}^2\leq \frac{16}{\pi r_v^2}\int_{B_{\mathrm{FS}}(z,r_v)}\|dv\|_{g_0}^2.\] Since $\int_{B_{\mathrm{FS}}(z,r_v)}\|dv\|_{g_0}^2\leq \pi$ and $r_v\in (\pi/24c,\pi/24)$, we have 
\[\|dv(z)\|_{g_0}\leq \frac{96c}{\pi}, \] 
for all $z\in \mathbb{CP}^1$. The constant $ c$ does not depend on $v$. 

Since $\mathbb{CP}^1\times M$ is compact, any Riemannian metric $g$ is comparable to $g_0:=(\omega_{\mathrm{FS}}\oplus\omega)(\cdot,J\cdot)$. So there exists $c_g>0$ such that 
\[\|\cdot\|_{g} \leq c_g\|\cdot\|_{g_0},\]
where $c_g$ varies continuously with $J$ and $g$ in the $C^\infty$-topology. Thus
\begin{equation}\label{conticonstant}
\|dv(z)\|_{g}\leq \frac{96c_gc}{\pi}:=C_{J,g},	
\end{equation}
for all $z\in \mathbb{CP}^1$. The constants $c_g$ and $c$ do not depend on $v$. 

The constant $k$ in 
\[c=\max\{e^{4k \pi},e^{\frac{18\pi^2}{\operatorname{injrad}(\mathbb{CP}^1\times M,g_0)}} , e^{2k\pi^2c^{-1}_{J,g_0}}\}\]
 varies continuously with the metric $g_0:=(\omega_{\mathrm{FS}}\oplus\omega)(\cdot,J\cdot)$, which in turn depends continuously on $J$ in the $C^\infty$-topology. By Theorem \ref{meanvaluek}, the constant $c_{J,g_0}>0$  depends continuously on $J$  in the $C^\infty$-topology. Therefore, the constant 
\[c=\max\{e^{4k \pi},e^{\frac{18\pi^2}{\operatorname{injrad}(\mathbb{CP}^1\times M,g_0)}} , e^{2k\pi^2c^{-1}_{J,g_0}}\}\]  varies continuously with $J$ in the $C^\infty$-topology.
The conclusion is that the constant 
\[C_{J,g}:=\frac{96c_gc}{\pi}\]
in (\ref{conticonstant}) varies continuously with $J$ and $g$ in  $C^\infty$-topology. This completes the outline of our proof.
\end{itemize} 
\subsection{Outline of the proof of Theorem \ref{compact1} via Gromov-Schwarz lemma}
Another approach to get a uniform $C^1$-bound on the moduli space in Theorem \ref{compact1} is to apply the monotonicity lemma, Lemma \ref{monochap1}, and the Gromov-Schwarz lemma, Lemma \ref{gschap1}, instead of mean value theorem for $J$-holomorphic curves as above. This argument goes as follows. We repeat the above steps until Step 03 to get 
\begin{equation}\label{cut122}
	\int_{B_{\mathrm{FS}}(z,r_v)}v^{*}(\omega_{\mathrm{FS}}\oplus\omega)\leq k \frac{2\pi^2}{\log(c)},
\end{equation}
 for any $c\geq \max\{e^{4k \pi}, e^{\frac{18\pi^2}{\operatorname{injrad}(\mathbb{CP}^1\times M,g_0)^2}} \}$  and some $r_v\in (\frac{\pi}{24c},\frac{\pi}{24})$ that depends on the map $v$. Recall that $c$ is arbitrary and does not depend on $v$.

 Let $\epsilon>0$ be the constant in Lemma \ref{gschap1}, and let $c_1$ and $c_2$ be the constants of Lemma \ref{monochap1} for the metric $g_0:=(\omega_{\mathrm{FS}}\oplus\omega)(\cdot,J\cdot)$. We prove that for 
\[c=\max\bigg\{e^{4k \pi}, e^{\frac{18\pi^2}{\operatorname{injrad}(\mathbb{CP}^1\times M,g_0)^2}}, e^{\frac{4k\pi^2}{c_1c_2^2}}, e^{\frac{8\pi^2}{\epsilon^2}\big(\sqrt{\frac{k}{c_1}}+1\big)^2}\bigg\}\]
the estimate (\ref{cut122}) and Lemma \ref{monochap1} imply the following: every $v$ admits some $r_v\in (\frac{\pi}{24c},\frac{\pi}{24})$ that depends on the map $v$ such that 
\[v(B_{\mathrm{FS}}(z,r_v))\subset B_{\varepsilon}(v(z)),\]
where $B_{\varepsilon}(v(z))$ denotes the ball of radius $\epsilon$ centered at $v(z)$ in $(\mathbb{CP}^1\times M, g_0)$. We then apply Lemma \ref{gschap1} to conclude that for all $z\in \mathbb{CP}^1$ we have 
\[\|dv(z)\|_{g_0}\leq C_{J,g_0}\]
for some constant $C_{J,g_0}>0$ that is continuous with respect to $J$ in $C^\infty$-topology and does not depend on $v$. For details, see Subsection \ref{proof-Gromov-Schwarz lemma}.
\section*{Acknowledgement}
The contents of this article are taken from my master's thesis at the Humboldt University of Berlin under the supervision of Klaus Mohnke. I wish to thank Klaus Mohnke for his guidance, Milica Dukic and Gorapada Bera for their useful comments which greatly improved the readability of this article. I received financial support from the Deutsche Forschungsgemeinschaft (DFG, German Research Foundation) under Germany's Excellence Strategy-The Berlin Mathematics Research Center MATH+ (EXC-2046/1, project ID: 390685689).

\section{Preliminaries}
\subsection{Symplectic manifolds}
\begin{definition}[Symplectic vector space]
	A symplectic vector space is a vector space $V$ together with a bi-linear $2$-form $\omega:V\times V\to \mathbb{R}$ which is  skew-symmetric and  non-degenerate, i.e., 
\begin{itemize}
	\item $\omega(v,w)=-\omega(w,v)$ for any two $v,w\in V$;
	\item for each $0\neq w\in V$, there exists $0\neq v\in V$ such that $\omega(w,v)\neq 0.$
\end{itemize} 
\end{definition}
\begin{definition}[Symplectic manifold]
	A symplectic manifold is a smooth manifold $X$ together a smooth differential $2$-form $\omega$ such that:
	\begin{itemize}	
		\item $(T_pX,\omega_p)$ is a symplectic vector space for every $p\in X$.
		\item $\omega$ is 
		de Rham closed, i.e., $d\omega=0$.
	\end{itemize}
\end{definition}
\begin{example}
Let $(x_1,y_1, \dots, x_n,y_n)$ be the coordinates on $\mathbb{R}^{2n}$. The $2$-form on $\mathbb{R}^{2n}$ defined by
	\[\omega_{\mathrm{std}}:=\sum_{i=1}^{n}dx_i\wedge dy_i\]
is a symplectic form. This is known as the standard symplectic form on $\mathbb{R}^{2n}$.
\end{example}
\begin{definition}[Symplectic embedding]
	Let $(X,\omega)$ and $(X',\omega')$ be two symplectic manifolds. A symplectic embedding of $(X,\omega)$ into $(X',\omega')$ is a smooth embedding $\psi:X\to X'$ such that $\psi^{*}\omega'=\omega$.
\end{definition}
\begin{definition}[Almost complex structure]
An almost complex structure on a smooth manifold $X$ is a map $X\ni p\to J_p:T_pX \to T_pX$ such that:
	\begin{itemize}	
		\item $J_p:T_pX \to T_pX$ is linear with $J_p^2:=J\circ J=-\operatorname{Id}$ for every  $p \in X$.
		\item For any smooth vector field $Y$ on $X$, $J(Y)$ is a smooth vector field on $X$.
	\end{itemize}
\end{definition}
\begin{definition}
	An almost complex manifold is a pair $(X,J)$, where $X$ is a smooth manifold and $J$ is an almost complex structure on $X$.
\end{definition}
\begin{definition}
A Riemann surface is an almost complex manifold of real dimension $2$. 
\end{definition}
Every almost complex structure on a $2$-dimensional manifold is integrable.
\begin{definition}
	Let $(X,\omega)$ be a symplectic manifold, and $J$ be an almost complex structure on $X$. We say $J$ is compatible with $\omega$ (or $J$ is $\omega$-compatible) if $\omega(\cdot,J\cdot)$ defines a Riemannian metric on $X$.
\end{definition}
The space of all almost complex structures on $X$ compatible with $\omega$ is denoted by $\mathcal{J}_c(X,\omega)$. The space $\mathcal{J}_c(X,\omega)$ is endowed with $C^{\infty}$-topology. It is well-known that $\mathcal{J}_c(X,\omega)$ is non-empty and contractible {\cite[Prop. 4.1.1]{MR3674984}.

\begin{example}
 Define 
	\[J_{\mathrm{std}}\bigg(\frac{\partial}{\partial x_i}\bigg):=\frac{\partial}{\partial y_i} \text{ and } J_{\mathrm{std}}\bigg(\frac{\partial}{\partial y_i}\bigg):=-\frac{\partial}{\partial x_i}.\]
One can verify that	$J_{\mathrm{std}}$ is an almost complex structure on $\mathbb{R}^{2n}$ compatible with $\omega_{\mathrm{std}}$ and $\omega_{\mathrm{std}}(\cdot, J_{\mathrm{std}}\cdot)$ is the standard Riemannian metric. 
\end{example}

\begin{proposition}\label{exist}
	Let $(X,\omega)$ be a symplectic manifold of dimension $2n$. Let $S$ be a compact submanifold of $X$ of the same dimension as $X$. Let $J_0$ be an almost complex structure on $S$ that is compatible with $\omega|_S$. There exists an almost complex structure $J$ on $X$ that is compatible with $\omega$ and agrees with $J_0$ on $S$, i.e., $J|_S=J_0$.
\end{proposition}
\begin{proof}
	We prove that there exists an extension of the metric $ g_0:=\omega (\cdot,J_0\cdot)$ to $X$. Then, we use the extended metric to extract an almost complex structure on each tangent space $T_pX$ which is compatible with $\omega_p$ and varies smoothly with respect to the base point $p$.
	
	Fix a point $x_0\in \partial S$ and choose a coordinate chart\ $(U,\varphi, x_1,x_2,...,x_{2n})$ around $x_0$ such that 
	\[\varphi:U\to \mathbb{B}^{2n}(1), \ \ \varphi(x_0)=0.\]
Here $\mathbb{B}^{2n}(1)$ denotes the unit ball centered at the origin in $\mathbb{R}^{2n}$. Since $S$ is a manifold with boundary, we can adjust $\phi$ so that
	\[\varphi(U\cap S)= \mathbb{B}^{+}:=\{(x_1,\dots,x_{2n})\in  \mathbb{B}^{2n}(1):x_{2n}\geq 0\}.\]
	Expressing $g_0$ in these coordinates, we get 
	\[g_0=\sum a_{ij}dx_i\otimes dy_j,\]
	where $a_{ij}$ are smooth real-valued functions on $U\cap S$. Composing these with $\varphi^{-1}$, we can think of these as real-valued smooths maps on $ \mathbb{B}^{+}$. 
	
	Let $\bar{a}_{ij}$ denote a smooth extension of $a_{ij}$ to $  \mathbb{B}^{-}=\{(x_1,\dots,x_{2n})\in \mathbb{B}^{2n}(1):x_{2n}\leq 0\}$. This is possible by  Whitney extension theorem \cite{Whitney1934}. This gives an extension of $g_0$ to  $U$
	\[g_0=\sum\bar{a}_{ij}dx_i\otimes dy_j.\]
	Cover $\partial S$ with finitely many charts $\{U_i\}$ and extend $g_0$ on each chart as above. Let $\{V_j\}$ be a cover of $X\setminus S$  by coordinate charts. Each $\{V_j\}$ carries a metric $g_j$ defined by 
	\[g_j:=\sum dx_i\otimes dy_i.\]
	Let $\{W_i\}$ be the cover of $X$ formed by $\operatorname{Int}(S)$, $\{U_i\}$ and $\{V_j\}$. Choose a partition of unity $\{\rho_i\}$ subordinate to $\{W_i\}$ and define 
	\[\bar{g}=\sum_{i}\rho_i g_i.\]
This is an extension of $g_0$ to $X$.
	
	Next, we construct $J$ with the desired properties. The construction goes point-wise as follows: fix $p\in X$, and let $J_p'$ be the endomorphism of the tangent space $T_pX$ defined by
	\[\bar{g}_p(J_p'\cdot,\cdot)=\omega(\cdot,\cdot).\]
By the non-degeneracy of $\omega$, we see that for any pair $v,w\in T_pX$
\[\bar{g}_p(J_p'v,w)=\omega(v,w)=-\omega(w,v)=-\bar{g}(v,J_p'w),\]
i.e., $J_p^{'*}=-J'_p$, where $J_p^{'*}$ denotes the adjoint of $J'_p$ with respect to $\bar{g}_p$. Hence $-J_p^{'2}$ is positive definite and symmetric. Let $K_p$ be the unique square root of $-J_p^{'2}$. Since $J_p'$ commutes with $K_p$ and $K_p$ is symmetric and positive
	definite, $p\to J_p:=K_p^{-1}J_p'$ is the required extension of the almost complex structure $J_0$. 
\end{proof}
\begin{definition}[Exponential Map]\label{exp}
	Geodesics on a Riemannian manifold $(X,g)$ solve Cauchy problems in local coordinates. For each $(p,v)\in TX$ there is a geodesic $\gamma:[0,\epsilon]\to X$ with $\gamma(0)=p$ and $\gamma'(0)=v$. For points $(p,v)\in TX$ for which $\gamma(1)$ makes sense, we define the exponential map as
	\[\operatorname{exp}_p(v)=\gamma(1).\]
\end{definition}
The map $\operatorname{exp}$ is defined on an open neighborhood of the zero section of $TX$, see \cite[Theorem 14.11]{LoringDG}. Moreover, for each point $p\in X$, $\operatorname{exp}_p$ is a diffeomorphism on some ball $B_r(0)\subseteq T_pX$ of radius $r$ onto its image.
\begin{definition}[Injectivity Radius]
	The injectivity radius of a Riemannian manifold $(X,g)$ at a point $p$ is defined by 
	\[\operatorname{injrad}(X,g, p):=\sup\big\{r:\operatorname{exp}_p|_{B_r(0)} \text{ is a diffeomorphism onto its image} \big\}.\]
	The injectivity radius of the Riemannian manifold $(X,g)$ is defined as 
	\[\operatorname{injrad}(X,g):=\inf_{p\in X} \operatorname{injrad}(X,g,p).\]
\end{definition}
\begin{proposition}\label{inj}
	For any compact Riemannian manifold $(X,\mu)$ we have $\operatorname{injrad}(X,g)>0$.
\end{proposition}
\begin{proof}
	We follow the argument in Hummel \cite{Hummel01}.	Each point $(p,v) \in TX$ has a neighborhood $V_{(p,v)}$ in $TX$ such that the map $G:V_{(p,v)}\to X\times X:(p,v)\to (p,\operatorname {exp}_p(v))$ is a diffeomorphism onto its image. The collection $\{G(V_{(p,v)})\}$ is an open cover of the diagonal in $X\times X$. Let $\epsilon>0$ be the Lebesgue number of this cover. For $p\in X$, denote by $B_\epsilon(p)$ the ball centered at $p$ and radius $\epsilon$ with respect to $g$. This means that for any $p\in X$ we have $B_{\epsilon}(p)\times B_{\epsilon}(p)\subseteq G(V_{(p,v)})$. Hence   $\operatorname{injrad}(X,g)>0.$
\end{proof}
\subsection{$J$-holomorphic curves and their moduli spaces}
\begin{definition}[$J$-holomorphic curve]
	Let $(X,J)$ be an almost complex manifold and $(S,j)$ be a Riemann surface.  A map $u:(S,j)\to (X,J)$ is called a $J$-holomorphic curve if its derivative $du:TX\to TX$ satisfies the equation
	\[du\circ j=J\circ du .\]
\end{definition}
\begin{remark}
	The differential $du$ splits  as
	\[du=\frac{1}{2}\big\{\underbrace{(du-J\circ du\circ j)}_{J\text{-linear}}+\underbrace{(du+J\circ du\circ j)}_{J\text{-antilinear}}\big\}.\]
	A map $u:(S,j)\to (X,J)$ is $J$-holomorphic if and only if the $J$-antilinear part vanishes, equivalently, the derivative $du$ is $J$-linear.
\end{remark}
\begin{remark}
	In case $(S,j)=(X,J)=(\mathbb{C},i)$, the equation above reduces to the usual Cauchy-Riemann equations in coordinates. Indeed, writing $du$ and $i$ in matrix forms and $u=(u_1,u_2)$, the equation $du\circ i=i\circ du$ can be written as 
	\[\bigg(\begin{matrix}
		\partial_{x}u_1 & \partial_{y}u_1\\
		\partial_{x}u_2 & \partial_{y}u_2
	\end{matrix}\bigg)\bigg(\begin{matrix}
		0 & -1\\
		1 & 0
	\end{matrix}\bigg)=\bigg(\begin{matrix}
		0 & -1\\
		1 & 0
	\end{matrix}\bigg)\bigg(\begin{matrix}
		\partial_{x}u_1 & \partial_{y}u_1\\
		\partial_{x}u_2 & \partial_{y}u_2
	\end{matrix}\bigg).\]
This is equivalent to the system of equations
	\[\partial_{x}u_1=\partial_{y}u_2,\  \partial_{x}u_2=-\partial_{y}u_1. \]
\end{remark}
A good introduction to the theory $J$-holomorphic curves is \cite{Wendl:2010aa}. If one wants to go deeper into the theory, one may continue with \cite{MacDuff}.

\begin{definition}[Simple $J$-holomorphic curves]
	Let $(S,j)$ be a closed Riemann surface and $(X,J)$ an almost complex manifold. A $J$-holomorphic curve $u:S\to X$ is called multiply covered if there is another closed Riemann surface $(S',j')$, a holomorphic branched curving $\phi:S\to S'$ and $J$-holomorphic curve $u':S'\to X$ such that 
	\[u=u'\circ \phi,\, \text{ and } \operatorname{degree}(\phi)>1.\]
A $J$-holomorphic curve is called simple if not multiply covered.
\end{definition}
\begin{definition}
	Let $(X,J)$ be an almost complex manifold and $(S,j)$ be any closed Riemann surface. Let $[S]$ be the fundamental class of $S$ representing the positive orientation of $S$. Every map $u:S\to X$ induces a map on the second homology
\[u_{*}:H_2(S, \mathbb{Z})\to H_2(X,\mathbb{Z}).\] 
Given $A\in H_2(X,\mathbb{Z})$, we say $u$ represents the homology class $A$ if $[u]:=u_{*}([S])=A$.
\begin{example}[Simple $J$-holomorphic curve]
Every curve in the moduli space (\ref{moduli}) is simple. To explain this, let $u:S\to X$ be a multiply covered $J$-holomorphic curve. Then by definition we can find a closed Riemann surface $(S',j')$, a holomorphic branched curving $\phi:S\to S'$ and $J$-holomorphic curve $u':S'\to X$ such that 
	\[u=u'\circ \phi,\, \text{ and } \operatorname{degree}(\phi)\in \mathbb{Z}_{\geq 2}.\]
This implies $u_{*}([S])=\operatorname{degree}(\phi)u'_{*}([S'])$. This is not possible if $u$ belongs to the moduli space (\ref{moduli}).
\end{example}
\end{definition}
Let $\operatorname{Aut}(S, j)$ denote the automorphism group of $(S,j)$, i.e.,  the group consisting of $j$-holomorphic map $g:S\to S$ that admits a $j$-holomorphic inverse $g^{-1}:S\to S$. The group $\operatorname{Aut}(\mathbb{CP}^1, i)$ is the group of M\"obius transformations.
\begin{definition}
Given an almost complex manifold  $(X,J)$ and a homology class $A\in H_{2}(X,\mathbb{Z})$. The moduli space of parameterized simple $J$-holomorphic spheres in $X$ representing the class $A$ is defined by 
\begin{equation*}
	\widehat{\mathcal{M}}^{s}(J,A):=\left\{
	\begin{array}{l}
		u:(\mathbb{CP}^1,i)\to (X,J),\\
		du\circ i=J\circ du ,\\
		u_*[S]=A\in H_2(X,\mathbb{Z}), \\
u \text{ is simple.}
\end{array}
	\right\}.
\end{equation*}
The moduli space of unparameterized simple $J$-holomorphic spheres in $X$ representing the class $A$ is defined by 
\begin{equation*}
\mathcal{M}^{s}(J,A):=\widehat{\mathcal{M}}^{s}(J,A)	\big/\sim,
\end{equation*}
where $u_1 \sim u_2$ if and only if $u_1=u_2\circ\varphi$ for some $\varphi \in \operatorname{Aut}(\mathbb{CP}^1,i)$.
\end{definition}
We topologize the moduli space $\widehat{\mathcal{M}}^{s}(J,A)$ with the $C^{\infty}$-topology and $\mathcal{M}^{s}(J,A)$ with the corresponding quotient topology. 
\begin{definition}
	Let $(X,\omega)$ be a symplectic manifold, and $J$ be an almost complex structure on $X$. We say $J$ is tamed by $\omega$ (or $J$ is $\omega$-tamed) if $\omega(v,Jv)>0$ for every non-zero tangent vector $v$.
\end{definition}
The space of all almost complex structures on $X$ tamed by $\omega$ is denoted by $\mathcal{J}_t(X,\omega)$. The space $\mathcal{J}_t(X,\omega)$ is endowed with $C^{\infty}$-topology. It is well-known that $\mathcal{J}_t(X,\omega)$ is nonempty and contractible {\cite[Prop. 4.1.1]{MR3674984}.
\begin{theorem}[{\cite[Theorem 3.1.5]{MacDuff}}]\label{trans0} Let $(X,\omega)$ be a closed symplectic manifold of dimension $2n$, and $A \in H_2(X,\mathbb{Z})$ be a homology class. There exists a subset $\mathcal{J}_{\mathrm{reg}}$ of $\mathcal{J}_t(X,\omega)$ such that:
	\begin{itemize}
		\item $\mathcal{J}_{reg}$ is a comeagre, i.e., it is a countable intersection of open dense subsets of  $\mathcal{J}_t(X,\omega)$.
		\item For every $J\in \mathcal{J}_{\mathrm{reg}}$, the moduli space $\mathcal{\widehat{M}}^s(J,A)$ is a smooth  oriented manifold of dimension
		\[2n+2c_1(A),\]
where $c_1$ denotes the first Chern number of the pullback bundle $(u^{*}TW, J)$ for a representative $u$ of the class $A$.
	\end{itemize}
\end{theorem}
\begin{theorem}[{\cite[Theorem 3.1.7]{MacDuff}}]\label{trans1} Let $(X,\omega)$ be a closed symplectic manifold of dimension $2n$.  Let  $\mathcal{J}_{t}(X,\omega)$ be the space of almost complex structures tamed by $\omega$, $A \in H_2(X,\mathbb{Z})$ be a homology class, and $\mathcal{J}_{reg}$ be set defined in Theorem \ref{trans0}. Given $J_0, J_1 \in \mathcal{J}_{reg}$, there exists a smooth path $\alpha:[0,1]\to \mathcal{J}_{t}(X,\omega)$ connecting $J_0$ to $J_1$ such that the moduli space
	\[\mathcal{\widehat{M}}^s(\alpha,A,\mathbb{CP}^1):=\big\{(t,u): t\in [0,1], u\in \mathcal{\widehat{M}}^s(\alpha(t),A)\big\}.\]
	is a smooth oriented manifold of dimension $2n+2c_1(A)+1$ with boundary 
	\[\partial \mathcal{\widehat{M}}^s(\alpha,A):=\mathcal{\widehat{M}}^s(J_0,A)\sqcup \mathcal{\widehat{M}}^s(J_1,A).\]
\end{theorem}
\begin{remark}\label{trans123}
Theorem \ref{trans0} and Theorem \ref{trans1} hold if we replace the space of $\omega$-tamed almost complex structures $\mathcal{J}_{t}(X,\omega)$ by the space of $\omega$-compitable almost complex structures $\mathcal{J}_{c}(X,\omega)$.
\end{remark}
There is a well-defined action of the group $\operatorname{Aut}(\mathbb{CP}^1,i)$ on the product $\widehat{\mathcal{M}}^{s}(J,A)\times \mathbb{CP}^1$, namely, for $\varphi \in \operatorname{Aut}(\mathbb{CP}^1,i)$ and $(u,z)\in \widehat{\mathcal{M}}^{s}(J,A)\times \mathbb{CP}^1$ define 
\[\varphi \cdot(u,z):=(u\circ \varphi,\varphi^{-1}(z))\in \widehat{\mathcal{M}}^{s}(J,A)\times \mathbb{CP}^1.\] 
We define
\[\mathcal{\widehat{M}}^{s}(J,A)\times_{\operatorname{Aut}(\mathbb{CP}^1)} \mathbb{CP}^1:=\widehat{\mathcal{M}}^s(J,A)\times \mathbb{CP}^1\big/\operatorname{Aut}(\mathbb{CP}^1, i).\]
\begin{definition}
The map defined by 
\[\operatorname{ev}_{J}:\mathcal{\widehat{M}}^{s}(J,A)\times_{\operatorname{Aut}(\mathbb{CP}^1)} \mathbb{CP}^1\to X,\, [( u,z)]\to u(z)\]
is called one-point evaluation map.
\end{definition}
The map $\operatorname{ev}_{J}$  connects the topology of the moduli spaces of $J$-holomorphic curves and that of $X$. It can be used to know much about the symplectic topology of $X$, see \cite{MacDuff}.
\begin{proposition}\label{evaluation}
The one-point evaluation map $\operatorname{ev}_{J}$ is well-defined and continuous in $C^\infty$-topology.
 If $\mathcal{\hat{M}}^s(J,A)$ is regular, i.e, if $J\in \mathcal{J}_{\mathrm{reg}}$, then $\operatorname{ev}_{J}$ is a smooth map.
\end{proposition}
\begin{proof}
	If $[(u,z)]=[(v,w)]$, then there exists $\varphi \in \operatorname{Aut}(\mathbb{CP}^1,i)$ such that $(u,z)=(u\circ \varphi,\varphi^{-1}(z))$. This implies $\operatorname{ev}_{J}[u,z]=\operatorname{ev}_{J}[v,w]$. So $\operatorname{ev}_{J}$ is well defined.
	
	With the topology on $\mathcal{\widehat{M}}^s(J,A)$ defined above, the evaluation map $\operatorname{ev_{J}}$ is continuous. Indeed, a $C^0$-small perturbation in $u$ brings small change in $u(z)$ which proves the continuity of  $\operatorname{ev}_{J}$ in the $C^0$-topology on $\mathcal{\widehat{M}}^s(J,A)$.
	
	If $J\in \mathcal{J}_{\mathrm{reg}}$, then $\mathcal{\widehat{M}}^s(J,A)$ is a smooth manifold by Theorem \ref{trans0}. We prove the map
	\[\operatorname{ev}_{J}:\mathcal{\widehat{M}}^s(J,A)\times \mathbb{CP}^1\to X, (u,z)\to u(z)\]
	is smooth and descends to a smooth map on the quotient $\mathcal{\widehat{M}}^s(J,A)\times_{\operatorname{Aut}(\mathbb{CP}^1)} \mathbb{CP}^1$.
	
	Let $U$ be an open neighborhood of the zero section in $TX$ such that exponential map $\operatorname{exp}:U\to W$ is a diffeomorphism onto its image. For a smooth map $u:\mathbb{CP}^1\to  X$  define\footnote{As a reference for Sobolev spaces of sections of vector bundles, we recommend {\cite[Appendix A.4]{Wendl:2016aa}}.}
\begin{equation*}
		W^{1,2}(u^*TU):=\left\{
		\begin{array}{l}
			\xi: \mathbb{CP}^1\to u^*TX,\\
			\xi \text{ is a section of the bundle $u^*TX$},\\
			\xi(z)\in U,\\
			\xi \text{ is } W^{1,2}\text{-regular}.
		\end{array}
		\right\}.
	\end{equation*}
$\big\{W^{1,2}(u^*TU),\operatorname{exp}_{u}\}_{u\in C^{\infty}(\mathbb{CP}^1, 	X)}$ is a smooth Banach manifold structure on 
\begin{equation*}
	W^{1,2}(\mathbb{CP}^1,X):=\left\{
	\begin{array}{l}
		\operatorname{exp}_u(\xi):\mathbb{CP}^1\to X,\\
		u\in C^{\infty}(\mathbb{CP}^1,X),\\
		\xi \in W^{1,2}(u^*TU).
	\end{array}
	\right\}.
\end{equation*}
The map $\operatorname{ev}_J$ extends to $W^{1,2}(\mathbb{CP}^1,X)$ on the obvious way. This extended  $\operatorname{ev}_J$ looks like the following in local coordinates for any fixed $z$:
	\[\operatorname{exp}_{u(z)}^{-1}\circ \operatorname{ev}_{J}\circ \operatorname{exp}_{u(z)}: W^{1,2}(u^*TU)\to T_{u(z)}X, \, \xi \to \xi(z).\]
	This is just taking an element in the Banach space of sections $ W^{1,2}(u^*TU)$  and evaluating it at $z$ into the Banach space $T_{u(z)}X$. This proves the smoothness of
	\[\operatorname{ev}_{J}:\mathcal{\widehat{M}}^s(J,A)\times \mathbb{CP}^1\to X, ( u,z)\to u(z)\]
	for fixed $z$. We leave it to the reader to complete the proof. 
\end{proof}
\begin{definition}
	A Hermitian manifold is a triple $(X,J,\mu)$ where $X$ is a smooth manifold, $J$ is an almost complex structure, and $\mu$ is a Riemannian metric such that 
	\[\mu(v,w)=\mu(Jv,Jw)\] 
	for all tangent vectors $v$ and $w$.
\end{definition}
\begin{definition}\label{areaform}
	Let $(S,j)$ be a Riemann surface and $(X,J,\mu)$  be a Hermitian manifold. The $\mu$-area of a  map $u:S\to X$ is defined by
	\[\operatorname{Area}_\mu(u):=\int_{S}\sigma_{u^{*}\mu},\]
where $\sigma_{u^{*}\mu}$ is the 2-form defined by
	\[\sigma_{u^{*}\mu}(v,w):=\bigg(\mu(du(v),du(v))\mu(du(w),du(w))-\mu(du(v),du(w))^{2}\bigg)^{\frac{1}{2}},\]
	for a positively  orientated vectors ${v,w}$ in any tangent space of $S$.
\end{definition}
\begin{proposition}\label{areacomp}
Let $(S,j,h)$ be a Riemann surface with a Hermitian metric $h$. Let $(X,J,\mu)$  be a Hermitian manifold. For every $J$-holomorphic curve $u:(S,j)\to (X,J)$ we have 
\[\operatorname{Area}_\mu(u)=\int_{S}\|du\|^2_{\mu}\operatorname{vol}_{h},\]
where $\operatorname{vol}_{h}:=\sigma_{\operatorname{Id}^{*}h}$ is the volume form on $S$ induced by $h$ and $\|du\|_{\mu}$ is the operator norm of the differential $du$ with respect to $h$ and $\mu$.
\end{proposition}
\begin{proof}
Every $J$-holomorphic curve $u:(S,j)\to (X,J)$ is a conformal map, i.e., $u^{*}\mu=f h$ for some smooth function $f:S\to \mathbb{R}$. For a non-zero tangent vector $v$ of $S$ we have 
\[f=\frac{u^{*}\mu(v,v)}{h(v,v)}=\frac{\mu(du(v),du(v))}{h(v,v)}.\]
The left hand of this equation does not depend on $v$, so
\[f=\sup_{v}\frac{\mu(du(v),du(v))}{h(v,v)}=\|du\|^2_{\mu}.\]
So we have $u^{*}\mu=\|du\|^2_{\mu} h$. Also note that $h(v,jv)=0=\mu(du(v),Jdu(v))=u^*\mu(v,jv)$. We conclude that
\[\sigma_{u^{*}\mu}=\|du\|^2_{\mu}\sigma_{\operatorname{Id}^{*}h}.\]
Thus 
\[\operatorname{Area}_\mu(u):=\int_{S}\sigma_{u^{*}\mu}=\int_{S}\|du\|^2_{\mu}\sigma_{\operatorname{Id}^{*}h}=\int_{S}\|du\|^2_{\mu}\operatorname{vol}_{h}.\qedhere\]
\end{proof}
\begin{definition}
	Let $(X,\omega)$ be a symplectic manifold and $(S,j)$ be any Riemann surface. The symplectic area  of a map $u:S\to (X,\omega)$ is defined by
	\[E(u):=\int_{S} u^{*}\omega.\]
\end{definition}
\begin{lemma}\label{comparison}
	Let $(X, \omega, J)$ be any symplectic manifold with and $\omega$-compatible almost of complex structure $J$ and let $(S,j)$ be a Riemann surface. For any smooth map $u:S\to X$ we have the following estimate:
	\[\operatorname{Area}_\mu(u):=\int_{S}\sigma_{u^{*}\mu}\geq \int_{S}u^{*}\omega=:E(u),\]
where $\mu=\omega(\cdot,J\cdot)$. The equality holds if $u$ is $J$-holomorphic. 
\end{lemma}
\begin{proof}
	Recall that $\sigma_{u^{*}\mu}$ is defined by 
	\[\sigma_{u^{*}\mu}(v,w):=\bigg(\mu(du(v),du(v))\mu(du(w),du(w))-\mu(du(v),d(w))^2\bigg)^{\frac{1}{2}},\]
	for any positively  orientated vectors ${v,w}$ in any tangent space of $S$. We prove that 
	\[\sigma_{u^{*}\mu}(v,w)\geq u^{*}\omega(v,w).\]
	This holds at those points where the derivative $du$ vanishes. So, we can assume $du$ nowhere vanishes. Then $u^{*}\mu$ is a well-defined metric on $S$. Let $v$ and $w'$ denote the orthonormalized version of $v,w$ with respect to the metric $u^{*}\mu$. One can see that
	\[u^{*}\omega(v,w)=u^{*}\omega(v,w')\] 
	and 
	\[\sigma_{u^{*}\mu}(v,w):=\bigg(\mu(du(v),du(v))\mu(du(w'),du(w'))\bigg)^{\frac{1}{2}}.\]
	\begin{equation*}
		\begin{split}
			u^{*}\omega(v,w)&=u^{*}\omega(v,w')  \\
			&=\omega(du(v),du(w'))\\
			&=\mu(Jdu(v),du(w')) \text{    (by definition of $\mu$)}\\
			&\leq \sqrt{\mu(Jdu(v),Jdu(v))\mu(du(w'),du(w'))} \text{    (Cauchy-Schwartz inequality)} \\
			&=\sqrt{\mu(du(v),du(v))\mu(du(w'),du(w'))} \\
			&=\sqrt{u^{*}\mu(v,v)u^{*}\mu(w',w')}\\
			&=\sigma_{u^{*}\mu}(v,w).
		\end{split}
	\end{equation*}

The vectors $jv$ and $w'$ are parallel with respect to the metric $u^{*}\mu$, so the Cauchy-Schwartz inequality applied to these two vectors is equality. Repeating the above steps with $u$ being $J$-holomorphic yields 
	\[\sigma_{u^{*}\mu}(v,w)= u^{*}\omega(v,w).\qedhere\]
\end{proof}
The following  is an easy corollary that follows from Proposition \ref{areacomp} and Lemma \ref{comparison}.
\begin{corollary}\label{comparison0}
	Let $(X, \omega, J)$ be any symplectic manifold with $\omega$-compatible almost complex structure $J$ and let $\mu$ be the Hermitian metric defined by $\mu:=\omega(\cdot,J\cdot)$. Let $(S,j, h)$ be a Riemann surface with a Hermitian metric $h$. Let $u:S\to X$ be $J$-holomorphic, then
	\[E(u)=\int_{S}u^{*}\omega=\int_{S}\|du\|^2_{\mu}\operatorname{vol}_{h}.\]
	In particular, $E(u)\geq 0$ and the equality holds if and only if $u$ is constant.
\end{corollary}
\subsection{Properties of $J$-holomorphic curves}
In this section, we list some important properties of $J$-holomorphic curves. These will be cited in Section \ref{mainproof} in our proofs of Theorem \ref{compact1}.
\begin{lemma}[Monotonicity lemma, cf. {\cite[Theorem 1.3]{Hummel01}}]\label{monochap1}
Let	$(S,j)$ be a compact Riemann surface with non-empty boundary. Let $(X, J, g)$ be a compact Hermitian manifold. For $p\in X$, let $B_r(p)$ denote the open ball of radius $r$ centered at $p$ in $(X,g)$. There exist constants $c_1, c_2>0$ that only depend on $(J, g)$ such that for every $J$-holomorphic curve $u:S\to X$ satisfying $u(\partial S)\cap B_r(u(s_0))=\emptyset$ for $s_0 \in S\backslash \partial S$ and $r\in (0,c_2)$ we have 
	\begin{equation*}\label{area}
	\operatorname{Area}_g	(u(S)\cap B_r(u(s_0)))\geq c_1r^2.
	\end{equation*}
\end{lemma}
We are also interested in the following special case of this lemma.
\begin{lemma}[cf. {\cite[Theorem I.4.1]{Zehmisch-Geiges}} ]\label{monoi}
	Let $(S,j)$ be a compact Riemann surface with non-empty boundary, and let $B^{2n}(r)$ be the ball of radius $\sqrt{r/\pi}$ centered at the origin in $\mathbb{R}^{2n}$. Every $J$-holomorphic curve  $u:S\to (\mathbb{R}^{2n},\omega_{\mathrm{std}}, J_{\mathrm{std}})$ with $u(s_0)=0$ for some $s_0 \in S\backslash \partial S$ and $u(\partial S)\cap  B^{2n}(r)=\emptyset$ satisfies
	\[\operatorname{Area}_{\mathrm{std}}\big(u(S)\cap B^{2n}(r)\big)=\int_{S}u^*\omega_{\mathrm{std}}\geq r.\]
\end{lemma}
\begin{lemma}[Gromov-Schwarz lemma, cf. {\cite[Corollary 1.2]{Hummel01}}]\label{gschap1}
	Let $(X, J, g)$ be any compact Hermitian manifold and $(D^2(1),i,\lambda)$ be the unit disk with the standard complex structure and a metric $\lambda$ conformally equivalent to the standard metric, i.e., $\lambda=h^2(dx^2+dy^2)$ for some function $h:D^2(1)\to \mathbb{R}$.
There exist positive constants $\varepsilon, C_{J,g}>0$ such that every $J$-holomorphic curve $u:(D^2(1), i)\to (X,J)$ with $u(D^2(1))\subseteq B_\varepsilon(p)$ for some $p\in X$ satisfies
	\[\|du(0)\|_{\lambda, g}\leq C_{J,g},\]
	where the constant $C_{J,g}>0$ only depends on $(J,g)$. Moreover, the constant $C_{J,g}$ varies continuously with respect to $J$ and $g$ in  $C^{\infty}$-topology.
\end{lemma}

\begin{theorem}[Mean value inequality]\label{meanvaluek}
Let $(X,J)$ be a compact almost complex manifold. Denote by $D^2(r)$ the disk of radius $r>0$ centered at the origin in $\mathbb{C}$. For any Riemannian metric $g$ on $X$, there exist positive constants $c_{J,g}, c_g>0$ such that for every $J$-holomorphic disk
	$u:D^2(r)\to X$ with 
	\[\int_{D^2(r)}\|du\|_g^2< c_{J,g}\]
	we have
	\[\|du(0)\|_g^2\leq \frac{16c_g}{\pi r^2}\int_{D^2(r)}\|du\|_g^2. \]
	Moreover, the constant $c_{J,g}$ depends continuously on $J$ and $g$ in the $C^\infty$-topology, and $c_g$ is continuous with respect to the metric $g$ in the $C^\infty$-topology. If $J$ preserves the metric $g$, i.e., $g(J\cdot, J\cdot)=g(\cdot, \cdot)$, then $c_g=1$.
\end{theorem}
Theorem \ref{meanvaluek} follows from the following slightly more general theorem.
\begin{theorem}[cf. {\cite[Prop. 4.1]{Zinger:2017aa}}]\label{meanvalue}
Let $(X,J,g)$ be a Hermitian manifold, possibly non-compact. Let $B_l(x)$ denote the ball of radius $l>0$  centered at a point $x$ in $ (X,g)$. There exists a continuous function $f_{J,g}:X\times \mathbb{R} \to(0,\infty)$ such that every  $J$-holomorphic map $u:D^2(r)\to X$ that satisfies 
	\[u(D^2(r))	\subseteq B_l(x)\  and \ \int_{D^2(r)}\|du\|_g^2< f_{J,g}(x,l),\]
also satisfies 
	\[\|du(0)\|_g^2\leq \frac{16}{\pi r^2}\int_{D^2(r)}\|du\|_g^2. \]
\end{theorem}
\begin{proof}[Proof of Theorem \ref{meanvaluek}]

Pick a Hermitian metric $g_0$ on $(X,J)$. Since $X$ is compact, $k:=\operatorname{diameter}(X,g_0)$ is finite and positive.  Define
	\[ c_{J,g_0}:=\min_{(x,r)\in X\times [0,k]} f_{J,g_0}(x,r)< \infty,\]
where $f_{J,g_0}$ is the function that appeared in Theorem \ref{meanvalue}. By Theorem \ref{meanvalue}, for any $J$-holomorphic disk $u:D^2(r)\to X$ satisfying   
\[\int_{D^2(r)}\|du\|_{g_0}^2< c_{J,g_0}\]
we have
\[\|du(0)\|_{g_0}^2\leq \frac{16}{\pi r^2}\int_{D^2(r)}\|du\|_{g_0}^2. \]
	
Since the manifold $X$ is compact, any Riemannian metric on $X$ is comparable to $g_0$. If  $g$ is any other metric on $X$, then one can find constants $c_g, c_{J,g}>0$ such 
for any $J$-holomorphic disk $u:D^2(r)\to X$ satisfying   
\[\int_{D^2(r)}\|du\|_{g}^2< c_{J,g}\]
we have
\[\|du(0)\|_{g}^2\leq \frac{16c_{g}}{\pi r^2}\int_{D^2(r)}\|du\|_{g}^2. \]
The comparability constant $c_{g}$ depends continuously on $g$ in $C^\infty$-topology. In the proof of Theorem \ref{meanvalue}, it can be observed that in case $X$ is compact, the constant $c_{J,g}>0$ depends continuously on $J$ and $g$ in the $C^\infty$-topology.
\end{proof}
The proof of Theorem \ref{meanvalue} is based on the following two lemmas.
\begin{lemma}\label{meanvlaue1} Let $w:D^2(r)\to \mathbb{R}$ be a non-negative $C^2$-function such that $-b\leq \Delta w$ for some constant $b>0$, where $\Delta$ denotes the Laplacian. Then 
	\[w(0)\leq \frac{br^2}{8}+\frac{1}{\pi r^2}\int_{D^2(r)}w.\]
	
\end{lemma}
\begin{proof}[Proof of Lemma \ref{meanvlaue1}]
	The function $v: D^2(r)\to \mathbb{R}$ defined by 
	\[v(s,t):=w(s,t)+\frac{b}{4}(s^2+t^2)\]
	is subharmonic, i.e., $0\leq\Delta v$. By the mean value inequality for sub-harmonic functions, we have 
\[w(0)=v(0)\leq\frac{1}{\pi r^2}\int_{D^2(r)}v=\frac{br^2}{8}+\frac{1}{\pi r^2}\int_{D^2(r)}w.\qedhere\]
\end{proof}
\begin{lemma}\label{meanvlaue23} Let $w:D^2(1)\to \mathbb{R}$ be a non-negative $C^2$-function such that $-w^2\leq \Delta w$, where $\Delta$ denotes the Laplacian. If
	\[\int_{D^2(1)}w<\frac{\pi}{8},\]
then
\[w(0)\leq\frac{8}{\pi}\int_{D^2(1)}w.\]
\end{lemma}
\begin{proof}[Proof of Lemma \ref{meanvlaue23}]
	Let $w:D^2(1)\to \mathbb{R}$ be a function that satisfies  $- w^2\leq \Delta w$ and 
	\[\int_{D^2(1)}w<\frac{\pi}{8}.\]
We use the Heinz trick (cf. {\cite[Page 87]{MacDuff}}) to prove that $w$ is subharmonic up to a quadratic form on some disk $D^2(r)$ contained in $D^2(1)$, then from the mean value inequality we get the estimate
	\[w(0)\leq\frac{8}{\pi r^2}\int_{D^2(r)}w.\]
	
	Define a function $f:[0,1]\to \mathbb{R}$ by 
	\[f(t)=(1-t)^2\max_{z\in \bar{D}^2(t)}w(z).\]
	Note that $f(0)=w(0)$ and $f(1)=0$. Let $t^*\in (0,1)$ and $z^* \in \bar{D}^2(t^*)$  be such that 
	\[f(t^*)=\sup_{t\in [0,1]}f(t), \text{ and } c:=w(z^*)=\sup_{z\in \bar{D}^2(t^*)}w(z). \]
	Let $\delta:=(1-t^*)/2$, and denote by $\bar{D}^2_\delta(z^*)$ the closed disk of radius $\delta$ centered at $z^*$ in $\mathbb{C}$. We can see that 
	\[\sup_{z\in \bar{D}^2_\delta(z^*)}w(z)\leq \sup_{z\in \bar{D}^2_{t^*+\delta}}w(z)\leq \frac{f(t^*+\delta)}{(1-(t^*+\delta))^2}=4 w(z^*).\]
So on the ball $\bar{D}^2_\delta(z^*)$ we have 
	\[\Delta w\geq -w^2=-16c^2.\]
By Lemma \ref{meanvlaue1}, we have 
	\begin{equation}\label{meanvalue3}
		c=w(z^*)\leq 2c^2t^2+\frac{1}{\pi t^2}\int_{D^2(1)}w.
	\end{equation}
for every $0< t\leq \delta$. This implies that $4c\delta^2\leq 1$. To see this, suppose  $4c\delta^2> 1$.  Then $\frac{1}{\sqrt{4c}}<\delta$. Choosing $t=\frac{1}{\sqrt{4c}}<\delta$ in inequality (\ref{meanvalue3}) gives
	\[\frac{\pi}{8}\leq \int_{D^2(1)}w\]
	which is a contradiction to our assumption that 
	\[\frac{\pi}{8}> \int_{D^2(1)}w.\]
So we must have  $4c\delta^2\leq1$. 

For $t=\delta$, the estimate (\ref{meanvalue3}) can be written as 
\[c+\frac{c}{2}(-4c\delta^2)\leq \frac{1}{\pi \delta^2}\int_{D^2(1)}w.
\]
As $-4c\delta^2\geq -1$, we obtain 
\[\frac{c}{2}\leq \frac{1}{\pi \delta^2}\int_{D^2(1)}w.\]
This implies
	\[w(0)=f(0)\leq f(t^*)=(1-t^*)^2c=4\delta^2c\leq \frac{8}{\pi}\int_{D^2(1)}w.\qedhere\]
\end{proof}
\begin{lemma}\label{meanvlaue2} Let $w:D^2(r)\to \mathbb{R}$ be a non-negative $C^2$-function such that $-b w^2\leq \Delta w$ for some constant $b>0$, where $\Delta$ denotes the Laplacian. If
	\[\int_{D^2(r)}w<\frac{\pi}{8b},\]
then
\[w(0)\leq\frac{8}{\pi r^2}\int_{D^2(r)}w.\]
\end{lemma}
\begin{proof}[Proof of Lemma \ref{meanvlaue2}]
Let $w:D^2(r)\to \mathbb{R}$ be a function that satisfies the hypothesis of Lemma  \ref{meanvlaue2}. Then the function $\bar{w}:D^2(1)\to \mathbb{R}$ defined by $\bar{w}(s,t):=br^2w(rs,rt)$ satisfies the hypothesis of Lemma \ref{meanvlaue23}. So 
\[w(0)=\frac{\bar{w}(0)}{br^2}\leq \frac{8}{\pi br^2}\int_{D^2(1)}\bar{w}= \frac{8}{\pi r^2}\int_{D^2(r)}w.\qedhere\]
\end{proof}
\begin{proof}[Proof of Theorem \ref{meanvalue}]
By Lemma $\ref{meanvlaue2}$, it is enough to prove that for any $J$-holomorphic curve $u:D^2(r)\to X $ with $u(D^2(r))\subseteq B_l(x)$ for some $x\in X$ and $l >0$ the function $\phi: D^2(r)\to (0,\infty)$  defined by 
	\[\phi(z)=\frac{1}{2}\|du(z)\|_g^2\]
satisfies the inequality
	\[\Delta \phi\geq -C_g(x,l) \phi^2 \]
	for some constant $C_g(x,l)>0$ which is continuous with respect to $x$ and $l$.	
	Let $z=s+i t$ denote the standard coordinates on $\mathbb{C}$. We denote by $u_t$ and $u_s$ the $t$ and $s$ partial derivatives of $u$ at  $z\in D^2(r)$, respectively. Let $\nabla$ be the Levi-Civita connection of $g$. Since $u$ is $J$-holomorphic and $J$ preserves $g$, we have
	\[\|u_s\|_g^2=\|u_t\|_g^2.\] 
	Note that
	\[\phi(z)=\frac{1}{2}\|du(z)\|_g^2=\|u_t\|_g^2=\|u_s\|_g^2.\]
	
	Since the Levi-Civita connection $\nabla$ of $g$ is $g$-compatible and torsion-free, we  have
	\[\frac{1}{2}\nabla_{tt}\|u_s\|_g^2=\|\nabla_t u_s\|_g^2+\langle \nabla_{tt}u_s,u_s \rangle_g=\|\nabla_t u_s\|_g^2+\langle \nabla_{t}\nabla_s u_t,u_s \rangle_g.  \]
	Similarly, 
	\[\frac{1}{2}\nabla_{ss}\|u_t\|_g^2=\|\nabla_s u_t\|_g^2+\langle \nabla_{s}\nabla_t u_s,u_t \rangle_g.\]
	Thus,
	\begin{equation}\label{cry}
		\frac{1}{2}\Delta\phi=\frac{1}{2}(\nabla_{ss}\|u_t\|_g^2+\nabla_{tt}\|u_s\|_g^2)=\|\nabla_t u_s\|_g^2+\|\nabla_s u_t\|_g^2+\langle \nabla_{s}\nabla_t u_s,u_t \rangle_g+\langle \nabla_{t}\nabla_s u_t,u_s \rangle_g.
	\end{equation}
	Since $u$ is $J$-holomorphic, $u_s=-Ju_t$. Therefore,
	\begin{equation*}
		\begin{split}
			\langle \nabla_s\nabla_t u_s,u_t\rangle_g&=-\langle \nabla_s\nabla_t Ju_t,u_t\rangle_g\\
			&=-\langle J \nabla_s\nabla_t u_t,u_t\rangle_g-\langle (\nabla_sJ)\nabla_t u_t,u_t\rangle_g-\langle \nabla_s((\nabla_t J)u_t),u_t\rangle_g\\
			&=-\langle  \nabla_s\nabla_t u_t,u_s\rangle_g-\langle (\nabla_sJ)\nabla_t u_t,u_t\rangle_g-\langle \nabla_s((\nabla_t J)u_t),u_t\rangle_g\\
		\end{split}
	\end{equation*}
	Putting this in equation (\ref{cry}), we have 
	\[\frac{1}{2}\Delta\phi=\|\nabla_s u_t\|_g^2+\|\nabla_t u_s\|_g^2+\langle R_g(u_t,u_s)u_t, u_s\rangle_g+\langle (\nabla_sJ)\nabla_t u_t,u_t\rangle_g-\langle \nabla_s((\nabla_t J)u_t),u_t\rangle_g,\]
where $R_g$ is the curvature tensor of the connection $\nabla$. Since $u(D^2(r))\in B_l(x)$, we observe that
	
	\[\| \langle R_g(u_t,u_s)u_t, u_s\rangle\|\leq  C_g(x,J)\|u_t\|_g^2\|u_s\|_g^2,\]
	for some constant $C_g(x,l)>0$.
	Also	
	\begin{equation*}
		\begin{split}
			\|\langle (\nabla_sJ)\nabla_t u_t,u_t\rangle_g\|&\leq  C_{g,J}(x,l)\|u_t\|_g\|u_s\|_g\|\nabla_t(Ju_s)\|_g\\
			&\leq C_{g,J}(x,l)\|u_t\|_g\|u_s\|_g(\|u_t\|_g\|u_s\|_g+\|\nabla_tu_s\|_g)\\
			&\leq(C_{g,J}(x,l)+C_{g,J}^2(x,l))\|u_t\|^2\|u_s\|^2+\|\nabla_tu_s\|_g^2.\\
		\end{split}
	\end{equation*}
	\begin{equation*}
		\begin{split}
			\|\langle \nabla_s((\nabla_t J)u_t),u_t\rangle_g\|&\leq  C_{g,J}(x,l)\|u_t\|_g^2(\|u_t\|_g\|u_s\|_g+\|\nabla_tu_s\|_g)\\
			&\leq C_{g,J}(x,l)\|u_t\|_g^3\|u_s\|_g+ C_{g,J}^2(x,l)\|u_t\|_g^4+\|\nabla_tu_s\|_g^2.\\
		\end{split}
	\end{equation*}
This gives
	\[\frac{1}{2}\Delta\phi\geq -C_g(x,r)(\|u_s\|_g^2\|u_t\|_g^2+\|u_s\|_g\|u_t\|_g^3+\|u_t\|_g^4)\geq-3C_g(x,l)\phi^2,\]
for some constant $C_g(x,l)>0$.
\end{proof}
\section{Proof of Theorem \ref{compact1}}\label{mainproof}
\subsection{Proof of Theorem \ref{compact1} via mean value inequality}\label{proof-mean-value inequality}
In this subsection, we present a proof of Theorem  \ref{compact1} based on the mean value inequality described in Theorem \ref{meanvaluek}. We deduce the proof from the following theorem.
\begin{theorem}\label{compact}	Let $(M, \omega)$ be a closed symplectic manifold of dimension $2n-2\geq 2$ with vanishing second homotopy group, i.e., $\pi_2 (M)=0$. Let $J\in \mathcal{J}_c(\mathbb{CP}^1\times M, \omega_{\mathrm{FS}}\oplus\omega)$  and consider the moduli space
	\begin{equation}\label{Moduli34}
	\mathcal{M}(J,[\mathbb{CP}^1\times\{\operatorname{pt}\}]):=\left\{
	\begin{array}{l}
		u:(\mathbb{CP}^1,i)\to (\mathbb{CP}^1\times M,J),\\
		du\circ i=J\circ du ,\\
		u_*[\mathbb{CP}^1]=[\mathbb{CP}^1\times\{\operatorname{pt}\}]\in H_2(\mathbb{CP}^1\times M,\mathbb{Z}).
\end{array}
	\right\}\bigg/\sim
\end{equation}
where $u_1 \sim u_2$ if and only if $u_1=u_2\circ\varphi$ for some $\varphi \in \operatorname{Aut}(\mathbb{CP}^1,i)$. Pick a Riemannian metric $g$ on $\mathbb{CP}^1\times M$. Each $[u]\in \mathcal{M}(J,[\mathbb{CP}^1\times\{\operatorname{pt}\}])$ admits a representative $v$ such that
	\[\|dv(z)\|_{g}\leq C_{J,g},\]
	for all $z\in \mathbb{CP}^1$ and some constant $C_{J,g}>0$ that only depends on $(g,J)$. Moreover, the constant $C_{J,g}$ varies continuously with $J$ and $g$ in the $C^\infty$-topology.
\end{theorem} 
\begin{proof}[Proof of Theorem \ref{compact1}]
Let $\{J_t\}_{t\in [0,1]}\subset \mathcal{J}_c(\mathbb{CP}^1\times M,\omega_{\mathrm{FS}}\oplus \omega)$ be a continuous path of $(\omega_{\mathrm{FS}}\oplus \omega)$-compatible almost complex structures. For each $t\in [0,1]$, by Theorem \ref{compact}, there exists $C_{J_t,g}>0$ such that every $[u]\in \mathcal{M}(J_t,[\mathbb{CP}^1\times\{\operatorname{pt}\}])$ admits a representative $v$ satisfying 
	\[\|dv(z)\|_{g}\leq C_{J_t,g}\]
	for all $z\in \mathbb{CP}^1$. The constant $C_{J_t,g}>0$ only depends on $(g,J_t)$ and varies continuously with $t\in [0,1]$. Since the interval $[0,1]$ is compact, we can choose $C_{J_t,g}$ to be uniform in $t$. 
	
The topology on the moduli space in Theorem \ref{compact1} is metrizable as a special case of {\cite[Theorem 5.6.6(ii)]{MacDuff}}. So compactness, in this case, is equivalent to sequential compactness. Given a sequence $\{[u_k]\}$ in the moduli space in Theorem \ref{compact1}, there exist a sequence $\{t_k\}$ in $[0,1]$ and a corresponding sequence  $\{J_{t_k}\}$ in $ \{J_t\}_{t\in [0,1]}$ such that $u_k$ is $J_{t_k}$-holomorphic. Since $[0,1]$ is compact, $\{t_k\}$ has a subsequence, still denoted by $\{t_k\}$, that converges to some $t_{\mathrm{lim}} \in [0,1]$. This implies the sequence $\{J_{t_k}\}$ $C^\infty$-converges to $J_{t_{\mathrm{lim}}}\in \{J_t\}_{t\in [0,1]}$ 
	because the family $\{J_t\}_{t\in [0,1]}$ is continuous in the $C^\infty$-topology. Moreover, $\{u_k\}$ has a uniform $C^0$-bound because the target manifold $\mathbb{CP}^1\times M$ is closed. Also, by the above discussion,  there exists $C>0$ such that (after re-parametrizing  $u_k$) we have
\[\|du_k(z)\|_{g}\leq C,\]
for all $z\in \mathbb{CP}^1$, $k\in \mathbb{Z}_{\geq 1}$. This $C^1$-bound implies a $C^\infty$-bound on the sequence $\{u_k\}$ by {\cite[Sec. 2.2.3]{Casim2014}}.  By Arzel$\grave{a}$-Ascoli Theorem, $u_k$ has a subsequence that $C^\infty$-converges to a $J_{t_{\mathrm{lim}}}$-holomorphic map $u:\mathbb{CP}^1\to \mathbb{CP}^1\times M$. Using $C^0$-convergence, the limit $u$ represents the class $[\mathbb{CP}^1\times\{\operatorname{pt}\}]$.
\end{proof}

\begin{proof}[Proof of Theorem \ref{compact}] Define $g_1$, $g_2$ and $g_3$ $\in  \operatorname{Aut}(\mathbb{CP}^1,i)$ by
	\[ \left\{
	\begin{array}{ll}
		g_1(z)=\lambda_{1} z,&\\
		g_2(z)=\frac{z+\lambda_{2}}{z\lambda_2+1},&\\
		g_3(z)=\frac{z+\lambda_{3}}{-\lambda_{3}z+1},&\\
	\end{array}
	\right. \]
	where $\lambda_1, \lambda_2, \lambda_3 \in \mathbb{C}$. Let $\pi_1:\mathbb{CP}^1\times M\to \mathbb{CP}^1$ and $\pi_2:\mathbb{CP}^1\times M\to M$ be the canonical projections. Observe that for any smooth map $u:\mathbb{CP}^1\to \mathbb{CP}^1\times M$ one has
	\begin{align*}
	E(u):=\int_{\mathbb{CP}^1} u^{*}(\omega_{\mathrm{FS}}\oplus \omega)&=\langle[u],\omega_{\mathrm{FS}}\oplus \omega\rangle
	\\
	&=\langle[\pi_1\circ u]+[\pi_2\circ u],\omega_{\mathrm{FS}}\oplus \omega\rangle\\
	&=\int_{\mathbb{CP}^1} (\pi_1\circ u)^{*}\omega_{\mathrm{FS}}+\int_{\mathbb{CP}^1} (\pi_2\circ u)^{*}\omega.	
\end{align*}
Since $\pi_{2}(M)=0$, $\int_{\mathbb{CP}^1} (\pi_2\circ u)^{*}\omega=0$. Also $[\pi_2\circ u]=m[\mathbb{CP}^1]$ where $m$ is the mapping degree of $\pi_1\circ u$ which is an integer. Therefore,
\[E(u)=\int_{\mathbb{CP}^1} (\pi_1\circ u)^{*}\omega_{\mathrm{FS}}=m\int_{\mathbb{CP}^1}\omega_{\mathrm{FS}}.\]
Also 
\[\int_{\mathbb{CP}^1}\omega_{\mathrm{FS}}=\int_{\mathbb{C}}\frac{dt \wedge ds}{(1+t^2+s^2)^2}=\lim_{r\to \infty}\int_{0}^{r}\int_{0}^{2\pi}\frac{\rho d\rho d\theta}{(1+\rho^2)^2}=\pi.\]
	So 
\[E(u)=m\int_{\mathbb{CP}^1}\omega_{\mathrm{FS}}=m\pi.\]
 This means the symplectic area of any smooth map $u:\mathbb{CP}^1\to \mathbb{CP}^1\times M$  is an integer multiple of $\pi$. In particular, any smooth map $u:\mathbb{CP}^1\to \mathbb{CP}^1\times M$ with symplectic area in the open interval $(-\pi,\pi)$  must have zero symplectic area.
	
	We have $m=1$ if $u$ represents the class $[\mathbb{CP}^1\times \{\operatorname{pt}\}]$. So every $u \in \mathcal{M}(J,[\mathbb{CP}^1\times\{\operatorname{pt}\}])$ has symplectic area equal to $\pi$, i.e.,  $E(u)=\pi$. Set $v:=u \circ g_1 \circ g_2 \circ g_3$  and choose $\lambda_1, \lambda_2$ purely real and $\lambda_3$ purely imaginary such that
	\[ \left\{
	\begin{array}{ll}
		E(v|_{D^2(1)})=\pi/2,&\\
		E(v|_{\operatorname{Re}(z)\geq 0})=\pi/2,&\\
		E(v|_{\operatorname{Imag}(z)\geq 0})=\pi/2,&\\
	\end{array}
	\right. \]
where $D^2(1)$ is the unit disk centered at the origin in $\mathbb{C}$ which corresponds to the lower hemisphere on $\mathbb{CP}^1$ under the stereographic projection. 
	
The point of the above rescaling is that we want to make the symplectic area distribution of $u$ uniform over $\mathbb{CP}^1$. After rescaling with $g_1$, which fixes the centers $0$ and $\infty$ of the lower and upper hemispheres, respectively, $u$ may have high symplectic area concentration along the equator. To handle this, we rescale with $g_2$ that fixes the centers $-1$ and $1$ of the left and right hemispheres, respectively. However, it is not enough; we may still have a high symplectic area at $i$ and $-i$. Therefore, we rescale by $g_3$. On each of the six hemispheres on $\mathbb{CP}^1$, the rescaled map $v:=u \circ g_1 \circ g_2 \circ g_3$ has symplectic area equal to $\pi/2$.
	
A few words on why such $g_1,g_2, g_3$ exist: the $\operatorname{Aut}(\mathbb{CP}^1,i)$ is six-dimensional as a smooth manifold. Roughly speaking, three dimensions out of six are taken by the rotations of $\mathbb{CP}^1$, which are useless for the type of rescaling we want. What remains is three dimensional, and hence one has the freedom of choosing up to three automorphisms, $g_1,g_2,g_3$ with which the map $u \circ g_1 \circ g_2 \circ g_3$ attains the above symplectic area distribution. 

For $z\in \mathbb{CP}^1$, denote the Fubini-Study disk of radius $\pi/24$ centered at $z$ by $B_\mathrm{FS}(z,\pi/24)$. Next we prove that for any given $c>1$ (independent of $v$) we have
\begin{equation*}
	l^2(v(\partial B_{\mathrm{FS}}(z,r_v)))\leq \frac{2\pi^2}{\log(c)},
\end{equation*}
for some $r_{v}\in (\pi/24c,\pi/24)$ that depends on $v$. Here $l$ denotes the length of the loop $v(\partial B_{\mathrm{FS}}(z,r_v))$ in $\mathbb{CP}^1\times M$ with respect to the metric $g_0=(\omega_{\mathrm{FS}}\oplus \omega)(\cdot,J\cdot)$. 

Let $z\in \mathbb{CP}^1$ and think of $S^1\times(0,\pi/24)$ conformally embedded annulus centered at $z$ so that it lies on the spherical disk $B_{\mathrm{FS}}(z, \pi/24)$ with $S^1\times \{\pi/24\}$ mapped to the boundary of $B_{\mathrm{FS}}(z, \pi/24)$. We get a map $v:S^1\times(0,\pi/24)\to \mathbb{CP}^1\times M$ which is $J$-holomorphic. By Corollary \ref{comparison0}, the symplectic area of $v|_{S^1\times(0,r)}$ for any $r\in (0,\pi/24)$  is given by
	\[A(r)=\int_{0}^{r}\int_0^{2\pi}\|dv\|_{g_0}^2\rho\, d\rho\, dt.\]
	Differentiating this with respect to $r$ gives
	\[A'(r)=\int_0^{2\pi}\|dv\|_{g_0}^2r  dt.\]
	By Cauchy--Schwarz inequality, we have 
	\[\bigg(\int_{0}^{2\pi}\|dv\|_{g_0}rdt\bigg)^2\leq \bigg(\int_{0}^{2\pi} r^2dt\bigg)  \bigg(\int_{0}^{2\pi}\|dv\|_{g_0}^2dt\bigg). \]
	This gives
	\[\frac{1}{2\pi r}\bigg(\int_{0}^{2\pi}\|dv\|_{g_0}rdt\bigg)^2\leq  \bigg(\int_{0}^{2\pi}\|dv\|_{g_0}^2rdt\bigg). \]
	Therefore,
	\[A'(r)\geq\frac{1}{2\pi r}\bigg(\int_0^{2\pi}\|dv\|_{g_0} r dt\bigg)^2= \frac{1}{2\pi r} l^2(v|_{S^1\times\{r\}}).\]
	Let $c>1$, integrating from $\pi/24c$ to $\pi/24$ we get
	\begin{equation*}
		\begin{split}
			A(\frac{\pi}{24})-A(\frac{\pi}{24c})&\geq \int_{\pi/24c}^{\pi/24}\frac{1}{2\pi r}l^2(v|_{S^1\times \{r\}})dr\\
			&\geq \frac{1}{2\pi}\min _{\pi/24c\leq r\leq\pi/24}l^2(v|_{S^1\times \{r\}})\int_{\pi/24c}^{\pi/24}\frac{1}{r}dr\\
			&=\frac{1}{2\pi}\log(c)\min _{\pi/24c\leq r\leq\pi/24}l^2(v|_{S^1\times \{r\}})\\
			&=\frac{\log(c)}{2\pi}l^2(v|_{S^1\times\{r_v\}}),\\
		\end{split}
	\end{equation*}	
for some $r_v\in (\frac{\pi}{24c},\frac{\pi}{24})$ that depends on the map $v$. This estimate  implies that for every $J$-holomorphic sphere $u\in \mathcal{M}(J,[\mathbb{CP}^1\times\{\operatorname{pt}\}])$, the rescaled version $v:=u \circ g_1 \circ g_2 \circ g_3$ satisfies
	\begin{equation}\label{in}
		l^2(v(\partial B_{\mathrm{FS}}(z,r_v)))\leq \frac{2\pi^2}{\log(c)},
	\end{equation}
for some $r_v\in (\frac{\pi}{24c},\frac{\pi}{24})$ that depends on the map $v$. Moreover, $c>1$ is arbitrary and does not depend on $v$.

Next we prove that for any $c\geq \max\{e^{4k_3 \pi}, e^{\frac{18\pi^2}{\operatorname{injrad}(\mathbb{CP}^1\times M,g_0)^2}} \}$, where $k_3>0$ is a constant that only depends on the metric $g_0$, we have
\[		\int_{B_{\mathrm{FS}}(z,r_v)}v^{*}(\omega_{\mathrm{FS}}\oplus\omega)\leq k_3 l^2(v|_{\partial B_{\mathrm{FS}}(z,r_v)}),
\]
where $r_v\in (\frac{\pi}{24c},\frac{\pi}{24})$ is defined by (\ref{in}). We start by proving  that $-v|_{\partial B_{\mathrm{FS}}(z,r_v)}:=v|_{\partial(\mathbb{CP}^1\setminus B_{\mathrm{FS}}(z,r_v))}$ admits a smooth extension $\Phi:\mathbb{CP}^1\setminus B_{\mathrm{FS}}(z,r_v)\to \mathbb{CP}^1\times M$ such that
	\[\int_{B_{\mathrm{FS}}(z,r_v)}v^{*}(\omega_{\mathrm{FS}}\oplus\omega)+\int_{\mathbb{CP}^1\setminus B_{\mathrm{FS}}(z,r_v)}\Phi^{*}(\omega_{\mathrm{FS}}\oplus\omega) \in (-\pi,\pi).\]
	
	Since $\mathbb{CP}^1\times M$ is compact, its injectivity radius $\operatorname{injrad}(\mathbb{CP}^1\times M,g_0:=(\omega_{\mathrm{FS}}\oplus \omega)(\cdot,J\cdot))$ is positive by Proposition \ref{inj}. Let $D^2(1)$ be the unit Euclidean $2$-disk centered at the origin. Choose an orientation-preserving diffeomorphism $\phi:D^2(1)\to \mathbb{CP}^1\setminus B_{\mathrm{FS}}(z,r_v)$. Let $\gamma:\partial D^2(1)\to \mathbb{CP}^1\times M$ denote the loop $v \circ \phi $. For $c\geq e^{\frac{18\pi^2}{\operatorname{injrad}(\mathbb{CP}^1\times M,g_0)^2}}$, the estimate  (\ref{in}) implies $l(\gamma)<\operatorname{injrad}(\mathbb{CP}^1\times M,g_0)/2$, so the image of the loop $v \circ \phi$ lies in some geodesic ball for every $v \in \mathcal{M}(J,[\mathbb{CP}^1\times\{\operatorname{pt}\}])$.
	
Define $\hat{\Phi}:D^2(1)\to \mathbb{CP}^1\times M$ by 
	\[\hat{\Phi}(re^{i\theta})=\exp_{\gamma(0)}(r\xi(\theta))\]
where $\xi(\theta)\in T_{\gamma(0)}\mathbb{CP}^1\times M$ is defined by 
	\[\exp_{\gamma(0)}(\xi(\theta))=\gamma(\theta).\]
	The map $\Phi:=\hat{\Phi}\circ \phi^{-1}:\mathbb{CP}^1\setminus B_{\mathrm{FS}}(z,r_v)\to \mathbb{CP}^1\times M$ is clearly a smooth extension of $v|_{\partial(\mathbb{CP}^1\setminus B_{\mathrm{FS}}(z,r_v))}$.
Moreover, observe that 
	\[\bigg|\frac{\partial \hat{\Phi}}{\partial r}\bigg|=\big|\xi(\theta)\big|=d(\gamma(0),\gamma(\theta))\leq l(\gamma).\]
Additionally
	\[\bigg|\frac{\partial \hat{\Phi}}{\partial \theta}\bigg|\leq k_1|\xi'(\theta)|\leq k_2|\gamma'|.\]
	Here, the constants $k_1$ and $k_2$ only depend on the metric $g_0$ on $\mathbb{CP}^1\times M$ and vary continuously with it in $C^\infty$-topology.
	This gives
\[\bigg|\int_{D^2(1)}\hat{\Phi}^{*}(\omega_{\mathrm{FS}}\oplus\omega)\bigg|=\bigg|\int_{0}^{2\pi}\int_{0}^{1}(\omega_{\mathrm{FS}}\oplus\omega)\bigg(\frac{\partial \hat{\Phi}}{\partial \theta},\frac{\partial \hat{\Phi}}{\partial r}\bigg)dr\, d\theta\bigg|\leq k_3 l^2(\hat{\Phi}|_{\partial D^2(1)}).\]
	Here, the constant $k_3>0$ only depends on $g_0$ and varies continuously with the metric $g_0$ in  the $C^\infty$-topology. We get 
	\[\bigg|\int_{\mathbb{CP}^1\setminus B_{\mathrm{FS}}(z,r_v)}\Phi^{*}(\omega_{\mathrm{FS}}\oplus\omega)\bigg|=\bigg|\int_{D^2(1)}\hat{\Phi}^{*}(\omega_{\mathrm{FS}}\oplus\omega)\bigg|\leq k_3 l^2(\hat{\Phi}|_{\partial D^2(1)})=k_3 l^2(v|_{\partial B_{\mathrm{FS}}(z,r_v)}).\]
We conclude that 
\[\bigg|\int_{\mathbb{CP}^1\setminus B_{\mathrm{FS}}(z,r_v)}\Phi^{*}(\omega_{\mathrm{FS}}\oplus\omega)\bigg|\leq k_3 l^2(v|_{\partial B_{\mathrm{FS}}(z,r_v)}).\]
For $c\geq e^{4k_3 \pi}$ in (\ref{in}) we get
	\[\bigg|\int_{\mathbb{CP}^1\setminus B_{\mathrm{FS}}(z,r_v)}\Phi^{*}(\omega_{\mathrm{FS}}\oplus\omega)\bigg|\leq k_3l^2(v|_{\partial B_{\mathrm{FS}}(z,r_v)})<\frac{\pi}{2}.\]
The Fubini-Study ball $B_{\mathrm{FS}}(z,\pi/24)$ of radius $\pi/24$ centered at $z$ lies in one of the six hemisphere for any $z\in \mathbb{CP}^1$. By our rescaling above, the symplectic area of $v$ on the ball $B_{\mathrm{FS}}(z,\pi/24)$ is strictly less than $\pi/2$. Since $r_v\in (\frac{\pi}{24c},\frac{\pi}{24})$, the symplectic area of $v|_{B_{\mathrm{FS}}(z,r_v)}$ is strictly smaller than $\pi/2$. Thus,  $v|_{B_{FS}(z,r_v)} \cup \Phi $ gives a sphere with symplectic area in the interval  $(-\pi,\pi)$ and hence zero by the observation we made in the beginning. Thus 
\[		\int_{B_{\mathrm{FS}}(z,r_v)}v^{*}(\omega_{\mathrm{FS}}\oplus\omega)=\bigg|\int_{\mathbb{CP}^1\setminus B_{\mathrm{FS}}(z,r_v)}\Phi^{*}(\omega_{\mathrm{FS}}\oplus\omega) \bigg|\leq k_3 l^2(v|_{\partial B_{\mathrm{FS}}(z,r_v)}).
\]
Combining this with (\ref{in}), we obtain that for any $c\geq \max\{e^{4k_3 \pi}, e^{\frac{18\pi^2}{\operatorname{injrad}(\mathbb{CP}^1\times M,g_0)^2}} \}$ we have 
\begin{equation}\label{cut}
	\int_{B_{\mathrm{FS}}(z,r_v)}v^{*}(\omega_{\mathrm{FS}}\oplus\omega)\leq k_3 \frac{2\pi^2}{\log(c)}.
\end{equation}
for some $r_v\in (\frac{\pi}{24c},\frac{\pi}{24})$ that depends on the map $v$. Here $c>1$ does not depend on $v$. The constant $k_3>0$ only depends on the metric $g_0$ and varies with it continuously in the $C^\infty$-topology.	
	
Let $c_{J,g_0}>0$ be the positive constant for which Theorem \ref{meanvaluek} holds. Choose $c=\max\{e^{4k_3 \pi},e^{\frac{18\pi^2}{\operatorname{injrad}(\mathbb{CP}^1\times M,g_0)^2}}, e^{2k_3\pi^2c^{-1}_{J,g_0}}\}$ in (\ref{cut}), then Corollary \ref{comparison0} and estimate (\ref{cut}) imply
	\[\int_{B_{\mathrm{FS}}(z,r_v)}\|dv\|_{g_0}^2=\int_{B_{\mathrm{FS}}(z,r_v)}v^{*}(\omega_{\mathrm{FS}}\oplus\omega)\leq k_3l^2(u|_{\partial B_{\mathrm{FS}}(z,r_v)})<c_{J,g_0}.\]
	By  Theorem \ref{meanvaluek}, we have 
	\[\|dv(z)\|_{g_0}^2\leq \frac{16}{\pi r_v^2}\int_{B_{\mathrm{FS}}(z,r_v)}\|dv\|_{g_0}^2.\] Since $\int_{B_{\mathrm{FS}}(z,r_v)}\|dv\|_{g_0}^2\leq \pi$ and $r_v\in (\pi/24c,\pi/24)$, we have 
	\[\|dv(z)\|_{g_0}\leq \frac{96c}{\pi}, \] 
	for all $z\in \mathbb{CP}^1$. The constant $ c$ does not depend on $v$. 
	
	Since $\mathbb{CP}^1\times M$ is compact, any Riemanian metric $g$ is comparable to $g_0:=(\omega_{\mathrm{FS}}\oplus\omega)(\cdot,J\cdot)$. So there exists $c_g>0$ such that 
	\[\|\cdot\|_{g} \leq c_g\|\cdot\|_{g_0},\]
	where $c_g$ varies continuously with $g$ and $J$ in $C^\infty$-topology. Thus
	\begin{equation}\label{conticonstantfinal}
		\|dv(z)\|_{g}\leq \frac{96c_g c}{\pi}:=C_{J,g},	
	\end{equation}
	for all $z\in \mathbb{CP}^1$. The constants $c_g$ and $c$ do not depend on $v$. 
	
	The constants $k_3$ and $\operatorname{injrad}(\mathbb{CP}^1\times M,g_0)$ in 
	\[c=\max\{e^{4k_3 \pi},e^{\frac{18\pi^2}{\operatorname{injrad}(\mathbb{CP}^1\times M,g_0)^2}}, e^{2k_3\pi^2c^{-1}_{J,g_0}}\}\]
	depend continuously on the metric $g_0:=(\omega_{\mathrm{FS}}\oplus\omega)(\cdot,J\cdot)$ which in turn depends continuously on $J$ in  the $C^\infty$-topology. By Theorem \ref{meanvaluek}, the constant $c_{J,g_0}>0$  is continuous  with respect to $J$  in the $C^\infty$-topology. Therefore, the constant 
	\[c=\max\{e^{4k_3 \pi},e^{\frac{18\pi^2}{\operatorname{injrad}(\mathbb{CP}^1\times M,g_0)^2}}, e^{2k_3\pi^2c^{-1}_{J,g_0}}\}\] is also continuous with respect to $J$ in  the $C^\infty$-topology.
	The conclusion is the constant 
	\[C_{J,g}:=\frac{96c_gc}{\pi}\]
	in (\ref{conticonstantfinal}) varies continuously with $J$ and $g$ in the $C^\infty$-topology. This completes the proof.
\end{proof}
\subsection{Proof of Theorem \ref{compact1} via Gromov-Schwarz lemma}\label{proof-Gromov-Schwarz lemma}
\begin{proof}
We repeat the above proof untill we arrive at the estimate (\ref{cut}). Let $\epsilon>0$ be the constant in Lemma \ref{gschap1}, and let $c_1$ and $c_2$ be the constants of Lemma \ref{monochap1} for the metric $g_0:=(\omega_{\mathrm{FS}}\oplus\omega)(\cdot,J\cdot)$. We prove that for 
\[c=\max\bigg\{e^{4k_3 \pi}, e^{\frac{18\pi^2}{\operatorname{injrad}(\mathbb{CP}^1\times M,g_0)^2}}, e^{\frac{4k_3\pi^2}{c_1c_2^2}}, e^{\frac{8\pi^2}{\epsilon^2}\big(\sqrt{\frac{k_3}{c_1}}+1\big)^2}\bigg\}\]
the estimate (\ref{cut}) and Lemma \ref{monochap1} imply that every $v$ admits some $r_v\in (\frac{\pi}{24c},\frac{\pi}{24})$ that depends on the map $v$ such that 
\[v(B_{\mathrm{FS}}(z,r_v))\subset B_{\varepsilon}(v(z)),\]
where $B_{\varepsilon}(v(z))$ denotes the ball of radius $\epsilon$ centered at $v(z)$ in $(\mathbb{CP}^1\times M, g_0)$. We then apply Lemma \ref{gschap1} to conclude that all $z\in \mathbb{CP}^1$ we have $\|dv(z)\|\leq C_{J,g_0}$, for some constant $C_{J,g_0}>0$ that is continuous with respect to $J$ in the $C^\infty$-topology.

For $c\geq e^{\frac{4k_3\pi^2}{c_1c_2^2}}$, estimate (\ref{cut})  implies
\begin{equation}\label{Gromoeste}
E(v|_{B_{\mathrm{FS}}(z,r_v)}):=\int_{B_{\mathrm{FS}}(z,r_v)}v^{*}(\omega_{\mathrm{FS}}\oplus\omega)<c_1 c_2^2.
\end{equation}

Let $d$ be the distance on $\mathbb{CP}^1\times M$ induced by the Riemannian metric $g_0:=(\omega_{\mathrm{FS}}\oplus \omega)(\cdot,J\cdot)$. For any $s\in B_{\mathrm{FS}}(z,r_v)$ we have
\begin{equation}\label{contra1}
d(v(s),v(\partial B_{\mathrm{FS}}(z,r_v)))\leq \sqrt{\frac{E(v|_{B_{\mathrm{FS}}(z,r_v)})}{c_1}}.
\end{equation}
	Indeed, if this is not the case, then for some $s \in B_{\mathrm{FS}}(z,r_v)$ and $r>\sqrt{\frac{E(v|_{B_{\mathrm{FS}}(z,r_v)})}{c_1}}$ we would have 
\begin{equation}\label{Gromoeste1}
	d(v(s),v(\partial B_{\mathrm{FS}}(z,r_v)))>r>\sqrt{\frac{E(v|_{B_{\mathrm{FS}}(z,r_v)})}{c_1}}.
\end{equation}

This implies that $v|_{B_{\mathrm{FS}}(z,r_v)}:B_{\mathrm{FS}}(z,r_v)\to \mathbb{CP}^1\times M$ passes through the center of the ball $B_r(v(s))$ and maps the boundary $\partial B_{\mathrm{FS}}(z,r_v)$ to the set-complement of $B_r(v(s))$. Also by (\ref{Gromoeste}) we can choose $r$ in (\ref{Gromoeste1}) so that $r<c_2$. Applying Lemma \ref{monochap1} we get 
	\[ \sqrt{\frac{E(v|_{B_{\mathrm{FS}}(z,r_v)})}{c_1}}\geq r.\]
It leads us to the contradiction 
	\[r>\sqrt{\frac{E(v|_{B_{\mathrm{FS}}(z,r_v)})}{c_1}}\geq r.\]
	
So Estimate (\ref{contra1}) must hold. This implies that for $s \in \partial B_{\mathrm{FS}}(z,r_v)$ the ball $B_{r}(v(s))$ of radius $r=\sqrt{\frac{E(v|_{B_{\mathrm{FS}}(z,r_v)})}{c_1}}+l(v(\partial B_{\mathrm{FS}}(z,r_v)))$ covers the image $v(B_{\mathrm{FS}}(z,r_v))$ and hence for any $s_1, s_2\in  B_{\mathrm{FS}}(z,r_v)$
\[d(v(s_1),v(s_2))\leq 2\bigg(\sqrt{\frac{E(v|_{B_{\mathrm{FS}}(z,r_v)})}{c_1}}+l(v(\partial B_{\mathrm{FS}}(z,r_v)))\bigg)\]
This  with (\ref{cut}) and (\ref{in}) imply
	\[d(v(s_1),v(s_2))\leq 2\pi \sqrt{\frac{2}{\log(c)}}\bigg(\sqrt{\frac{k_3}{c_1}}+1\bigg).\]
	
Since $c\geq e^{\frac{8\pi^2}{\epsilon^2}\big(\sqrt{\frac{k_3}{c_1}}+1\big)^2}$, we get
\[d(v(s_1),v(s_2))\leq \epsilon,\]
for any $s_1, s_2\in  B_{\mathrm{FS}}(z,r_v)$. This implies 
	\[v|_{B_{\mathrm{FS}}(z,r_v)}:B_{\mathrm{FS}}(z,r_v)\to B_{\varepsilon}(v(z)).\]
The Gromov-Schwarz lemma, Lemma \ref{gschap1}, implies
\[\|dv(z)\|_{g_0}\leq C_{J,g_0},\]
for all $z\in \mathbb{CP}^1$ and some constant $C_{J,g_0}>0$ that does not depend on $v$ and varies continuously with $g_0$ and $J$ in the $C^\infty$-topology.

Since $\mathbb{CP}^1\times M$ is compact, any Riemannian metric $g$ is comparable to $g_0:=(\omega_{\mathrm{FS}}\oplus\omega)(\cdot,J\cdot)$. So there exists $c_g>0$ such that 
	\[\|\cdot\|_{g} \leq c_g\|\cdot\|_{g_0},\]
	where $c_g$ varies continuously with $g$ and $J$ in the $C^\infty$-topology. Thus
\[\|dv(z)\|_{g}\leq c_g C_{J,g_0}\]
for all $z\in \mathbb{CP}^1$. 

This gives a uniform $ C^1$-bound on the module space in Theorem \ref{compact} in terms of a constant that varies continuously with the almost complex structure $J$ for a given fixed Riemannian metric $g$. Higher jets of pseudo-holomorphic curves can be turned into pseudo-holomorphic curves in a suitable target manifold to which Gromov-Schwarz lemma can be applied, see {\cite[Chapter III]{Hummel01}}. So, we can inductively apply the above argument to the higher jets of curves in the moduli space of Theorem \ref{compact} and get a uniform bound on higher jets of every order. The compactness of the moduli spaces in Theorem \ref{compact}  and Theorem \ref{compact1} then follow from Arzel$\grave{a}$-Ascoli Theorem. This proof does not rely on elliptic regularity results for Cauchy-Riemann equation.\qedhere
\end{proof}

We observe that each of the moduli spaces defined by (\ref{Moduli34}) and (\ref{moduli}) carries the minimal positive symplectic area, and this is very essential to our proofs presented in the above two sections. By apply our arguments from either Subsection \ref{proof-mean-value inequality} or Subsection \ref{proof-Gromov-Schwarz lemma}, we obtain a proof of the following more general theorem.
\begin{theorem}[cf. Theorem \ref{compact1}]
Let $(X, \omega)$ be any closed symplectic manifold. Let $Z\subseteq H_2(X,\mathbb{Z})$ be the image of the Hurewicz map $\pi_2(X)\to H_2(X,\mathbb{Z})$. Let $A\in Z$ be a homology class of the minimal positive symplectic area in $Z$, i.e, 
\[0<\int_A\omega=\inf\bigg\{\int_W\omega>0:W\in Z\subseteq H_2(X,\mathbb{Z})\bigg\}.  \]

 Let $I$ be a compact topology space and $\{J_t\}_{t\in I}\subset \mathcal{J}_c(X,\omega)$ be a continous familiy of $\omega$-compatible almost complex structures.  Define
	\begin{equation*}
	\mathcal{M}(\{J_t\}_{t\in I},A):=\left\{(t,u):
	\begin{array}{l}
		t\in I,\\
		u:(\mathbb{CP}^1,i)\to (X,J_t),\\
		du\circ i=J_t\circ du ,\\
		u_*[\mathbb{CP}^1]=A.
	\end{array}
	\right\}\bigg/\sim
\end{equation*}
where $u_1 \sim u_2$ if and only if $u_1=u_2\circ\varphi$ for some $\varphi \in \operatorname{Aut}(\mathbb{CP}^1,i)$. The moduli space $\mathcal{M}(\{J_t\}_{t\in I},A)$  is compact in the quotient topology coming from  $I\times C^\infty(\mathbb{CP}^1,X)$. 
\end{theorem}

\section{Proof of Theorem \ref{ab}}
In this section, we explain a proof of Theorem \ref{ab}. Assuming the hypothesis  of Theorem \ref{ab}, the idea of the proof is to prove that for generic $J\in \mathcal{J}_c(\mathbb{CP}^1\times M, \omega_{\mathrm{FS}}\oplus\omega)$ the evaluation map 
\[\operatorname{ev}_{J}: \widehat{\mathcal{M}}(J,[\mathbb{CP}^1\times\{\operatorname{pt}\}] )\times_{\operatorname {Aut}(\mathbb{CP}^1)} \mathbb{CP}^1\to \mathbb{CP}^1\times M\] has degree $1$ mod $2$. This. in other words, means that Theorem \ref{ab} holds for generic choice of $J$ in $\mathcal{J}_c(\mathbb{CP}^1\times M, \omega_{\mathrm{FS}}\oplus\omega)$ and generic choice of $p$ in $ \mathbb{CP}^1\times M$. Having established this, one can then construct a sequence $J_n$ that $C^{\infty}$-converges to the given $J$ and another sequence $p_n\in \mathbb{CP}^1\times M$  converges to $p$. Corresponding to these two sequences, one can choose elements $[(u_n,J_n,z_n)]\in \widehat{\mathcal{M}}(J_n, [\mathbb{CP}^1\times\{\operatorname{pt}\}])\times_{\operatorname{Aut}(\mathbb{CP}^1)} \mathbb{CP}^1$ that admits a convergent sequence by Theorem \ref{compact1}. The limit of this subsequence is the required curve passing through $p$. We achieve this in a sequence of lemmas below. We follow the presentations given in \cite{MacDuff} and \cite{Wendl:2010aa}.
\begin{lemma}\label{split0}
	Let $J_M$ be an $\omega$-compatible almost complex structure on $(M,\omega)$.  For the split almost complex structure $i\oplus J_M$ on $\mathbb{CP}^1\times M$, the moduli space  $\mathcal{\widehat{M}}(i\oplus J_M,[\mathbb{CP}^1\times\{\operatorname{pt}\}])$ is a finite-dimensional smooth manifold and the evaluation map 
	\[\operatorname{ev}_{i\oplus J_M}:\widehat{\mathcal{M}}(i\oplus J_M,[\mathbb{CP}^1\times\{\operatorname{pt}\}])\times_{\operatorname{Aut}(\mathbb{CP}^1,i)} \mathbb{CP}^1\to \mathbb{CP}^1\times M\]
	 is a diffeomorphism.
\end{lemma}
\begin{proof}
	A map $u:\mathbb{CP}^1\to \mathbb{CP}^1\times M$, written as $u=(u_1,u_2)$, is $i\oplus J_M$-holomorphic if and only $u_1:\mathbb{CP}^1 \to \mathbb{CP}^1$  is  $i$-holomorphic and $u_2:\mathbb{CP}^1 \to M$ is $J_M$-holomorphic. Since $\pi_{2}(M)=0$, the map $u_2$ has zero symplectic area, i.e.,
	\[\int_{\mathbb{CP}^1}u_2^{*}\omega=0.\] By Corollary \ref{comparison0},  $u_2$ is a constant map.

Since $u$ represents the homology class $[\mathbb{CP}^1\times\{\operatorname{pt}\}]$,  $u_1$ represents the homology class $[\mathbb{CP}^1]$. This means $u_1:\mathbb{CP}^1 \to \mathbb{CP}^1$  has mapping degree equal to $1$ and hence  $u_1\in \operatorname{Aut}(\mathbb{CP}^1,i)$. We conclude that
	\[\mathcal{\widehat{M}}(i\oplus J_M,[\mathbb{CP}^1\times\{\operatorname{pt}\}])=\big\{(\varphi,m): \varphi\in \operatorname{Aut}(\mathbb{CP}^1,i), m\in M\big\},\]
	where $(\varphi ,m)$ is interpreted as a $i\oplus J_M$-holomorphic map $u^{\varphi}_m:\mathbb{CP}^1\to \mathbb{CP}^1\times M$ defined by $u^{\varphi}_m(z):=(\varphi(z),m)$.
	
The pull-back complex bundle $((u^{\varphi}_m)^*T(\mathbb{CP}^1\times M),i\oplus J_M)$ over $\mathbb{CP}^1$ splits as 
\[(u^{\varphi}_m)^*T(\mathbb{CP}^1\times M)=(\varphi^*T\mathbb{CP}^1,i)\oplus (E,J_M)\]
where $E\to \mathbb{CP}^1$ is the trivial bundle of complex rank $n-1$ whose fiber at each $z\in \mathbb{CP}^1$ is $(T_mM,J_M)$. Since $(\varphi^*T\mathbb{CP}^1,i)\simeq (T\mathbb{CP}^1,i)$, we have 
\[(u^{\varphi}_m)^*T(\mathbb{CP}^1\times M)=(\varphi^*T\mathbb{CP}^1,i)\oplus (E,J_M)\simeq (T\mathbb{CP}^1,i)\oplus (E,J_M)\]
By {\cite[Theorem 2.7.1]{MR3674984}}, the first Chern number of $(u^{\varphi}_m)^*T(\mathbb{CP}^1\times M)$ can be computed as follows
\[c_1((u^{\varphi}_m)^*T(\mathbb{CP}^1\times M),i\oplus J_M)=c_1(T\mathbb{CP}^1,i)+c_1(E,J_M)= \chi(\mathbb{CP}^1)+0=2=c_1([\mathbb{CP}^1\times\{\operatorname{pt}\}]).\]
We will need this computation latter in our argument.
 
\textbf{Smoothness:} we show that $\mathcal{\widehat{M}}(i\oplus J_M,[\mathbb{CP}^1\times\{\operatorname{pt}\}])$ is a smooth finite-dimensional manifold. Let $W^{1,3}(\mathbb{CP}^1,\mathbb{CP}^1\times M)$ denote the space of functions $u:\mathbb{CP}^1\to \mathbb{CP}^1\times M$ that are of Sobolev class $W^{1,3}$ and represent the homology class $[\mathbb{CP}^1\times\{\operatorname{pt}\}]$. For $u\in  W^{1,3}(\mathbb{CP}^1,\mathbb{CP}^1\times M)$, let $\overline{\text{Hom}}_{\mathbb{C}}(T\mathbb{CP}^1, (u^{*}T(\mathbb{CP}^1\times M)))$ be the bundle of complex-antilinear 1-forms on $\mathbb{CP}^1$ with values in the complex vector bundle $(u^{*}T(\mathbb{CP}^1\times M),i\oplus J_M)$. Let $L^3(\overline{\text{Hom}}_{\mathbb{C}}(T\mathbb{CP}^1, u^{*}T(\mathbb{CP}^1\times M)))$ denote the space of $L^3$-sections of 
\[\overline{\text{Hom}}_{\mathbb{C}}(T\mathbb{CP}^1, (u^{*}T(\mathbb{CP}^1\times M))).\] 
One can prove that 
\[\Lambda:=\bigcup_{u\in W^{1,3}(\mathbb{CP}^1,\mathbb{CP}^1\times M) } L^3(\overline{\text{Hom}}_{\mathbb{C}}(T\mathbb{CP}^1, u^{*}T(\mathbb{CP}^1\times M))\]
is a smooth Banach bundle with base $W^{1,3}(\mathbb{CP}^1,\mathbb{CP}^1\times M)$ and fiber \[L^3(\overline{\text{Hom}}_{\mathbb{C}}(T\mathbb{CP}^1, u^{*}T(\mathbb{CP}^1\times M))\]
over $u\in W^{1,3}(\mathbb{CP}^1,\mathbb{CP}^1\times M)$, see {\cite[Chapter 3]{MacDuff}} for detailed analysis.

Consider the non-linear Cauchy-Riemann operator \[\bar{\partial}: W^{1,3}(\mathbb{CP}^1,\mathbb{CP}^1\times M)\to \Lambda\]
defined by $\bar{\partial}(u)=du+(i\oplus J_M)\circ du\circ i$. Note that 
\[\mathcal{\widehat{M}}(i\oplus J_M,[\mathbb{CP}^1\times\{\operatorname{pt}\}])=\bar{\partial}^{-1}(0),\]
where $0$ denotes the $0$-section of $\Lambda$. For every $u\in \bar{\partial}^{-1}(0)$, the linearization of $\bar{\partial}$ at $u$, denoted by $D_{u}\bar{\partial}$, is a real linear Cauchy-Riemann operator 
	\[D_{u}\bar{\partial}:W^{1,3}(\mathbb{CP}^1, u^{*}T(\mathbb{CP}^1\times M))\to L^3(\overline{\text{\text{Hom}}}_{\mathbb{C}}(T\mathbb{CP}^1, u^{*}T(\mathbb{CP}^1\times M)),\]
where $W^{1,3}(\mathbb{CP}^1, u^{*}T(\mathbb{CP}^1\times M)$ denotes the space of $W^{1,3}$-sections of the pullback bundle $u^{*}T(\mathbb{CP}^1\times M)\to \mathbb{CP}^1$, for details see {\cite[Section 3.1]{MacDuff}}. By {\cite[Theorem 3.1.8]{Wendl:2010aa}}, the operator $D_{u}\bar{\partial}$ is Fredholm of index
\[\operatorname{ind}(D_{u}\bar{\partial})=n\chi(\mathbb{CP}^1)+2c_1([u])=2n+2c_1([\mathbb{CP}^1\times\{\operatorname{pt}\}])=2n+4.\]

To show that $\mathcal{\widehat{M}}([\mathbb{CP}^1\times\{\operatorname{pt}\}], i\oplus J_M)=\bar{\partial}^{-1}(0)$ is a smooth manifold, by the implicit function theorem, it is enough to prove that $\hat{\partial}$ is traverse to the zero section in $\Lambda$, or equivalently, $D_{u}\bar{\partial}$ is surjective for every $u\in \bar{\partial}^{-1}(0)$. The dimension of $\mathcal{\widehat{M}}([\mathbb{CP}^1\times\{\operatorname{pt}\}], i\oplus J_M)$ as a smooth manifold is then given by the Fredholm index of $D_{u}\bar{\partial}$ which is $2n+4$ by the above calculation.
	
Recall that we have the splitting 
\[((u^{\varphi}_m)^*T(\mathbb{CP}^1\times M),i\oplus J_M)=(\varphi^*T\mathbb{CP}^1,i)\oplus (E,J_M).\]
This gives the splittings 
\[W^{1,3}(\mathbb{CP}^1, (u^{\varphi}_m)^*T(\mathbb{CP}^1\times M))=W^{1,3}(\mathbb{CP}^1, \varphi^{*}T\mathbb{CP}^1)\oplus W^{1,3}(\mathbb{CP}^1,E)\]
and 
\[L^3(\overline{\text{\text{Hom}}}_{\mathbb{C}}(T\mathbb{CP}^1, (u^{\varphi}_m)^*T(\mathbb{CP}^1\times M)))=L^3(\overline{\text{\text{Hom}}}_{\mathbb{C}}(T\mathbb{CP}^1, \varphi^{*}T\mathbb{CP}^1))\oplus L^3(\overline{\text{\text{Hom}}}_{\mathbb{C}}(T\mathbb{CP}^1, E)).\]
With respect to these splittings, the linearized Cauchy-Riemann operator under discussion can be written as
\[D_{u^{\varphi}_m}\bar{\partial}=\begin{pmatrix}
D_1 & 0 \\
0 & D_2 
\end{pmatrix},\]
where $D_1$ and $D_2$ real-linear Cauchy-Riemann type operators 
\[D_1:W^{1,3}(\mathbb{CP}^1, \varphi^{*}T\mathbb{CP}^1)\to L^3(\overline{\text{\text{Hom}}}_{\mathbb{C}}(T\mathbb{CP}^1, \varphi^{*}T\mathbb{CP}^1))\]
and 
\[D_2: W^{1,3}(\mathbb{CP}^1,E)\to L^3(\overline{\text{\text{Hom}}}_{\mathbb{C}}(T\mathbb{CP}^1, E)).\]
	
Note that $c_1(\varphi^{*}T\mathbb{CP}^1,i)=\chi(\mathbb{CP}^1)=2$ and $c_1(E,J_M)=0$ because $E$ is a trivial bundle. The linear Cauchy-Reimann operators $D_1$ and $D_2$ are both surjective by {\cite[Lemma 3.3.2]{MacDuff}} which states the following: let $E\to \mathbb{CP}^1$ be any complex vector bundle of complex rank $n$ such that $E=\oplus_{k=1}^{m}E_i$, where $E_i$ are sub-bundles of $E$. Let $\Gamma(E)$ denote the space section of $E$ with a suitable regularity. Let  \[D:\Gamma(E)\to \Gamma(\overline{\operatorname{Hom}}_{\mathbb{C}}(T\mathbb{CP}^1, E))\]
be a real-linear Cauchy-Reimann operator such that $E_i$ are $D$-invariant.  Then $D$ is surjective if and only if $c_{1}(E_k/E_{k-1})>-2$ for all $k$. Applying this to $D_{u^{\varphi}_m}\bar{\partial}$, the above discussion implies $D_{u^{\varphi}_m}\bar{\partial}$ is surjective. Hence, $\mathcal{\hat{M}}(i\oplus J_M,[\mathbb{CP}^1\times\{\operatorname{pt}\}])$ is a smooth manifold of dimension $2n+4$. 

The quotient 
\[\widehat{\mathcal{M}}(i\oplus J_M,[\mathbb{CP}^1\times\{\operatorname{pt}\}])\times_{\operatorname{Aut}(\mathbb{CP}^1,i)} \mathbb{CP}^1.\]
is a smooth manifold of dimension $2n$. Also, observe that
\[\widehat{\mathcal{M}}(i\oplus J_M,[\mathbb{CP}^1\times\{\operatorname{pt}\}])\times_{\operatorname{Aut}(\mathbb{CP}^1,i)} \mathbb{CP}^1=\big\{(\operatorname{Id},m, z): m\in M, z\in \mathbb{CP}^1\big\}.\]
So evaluation map
\[\operatorname{ev}_{i\oplus J_M}:\widehat{\mathcal{M}}(i\oplus J_M,[\mathbb{CP}^1\times\{\operatorname{pt}\}])\times_{\operatorname{Aut}(\mathbb{CP}^1,i)} \mathbb{CP}^1\to \mathbb{CP}^1\times M\]
takes the form 
\[\operatorname{ev}_{i\oplus J_M}(\operatorname{Id},m, z)=(z,m).\]
which is clearly a diffeomorphism.
\end{proof}
\begin{lemma}\label{split2}
	There exists a subset $\mathcal{J}_{\mathrm{reg}}$ of $\mathcal{J}_c(\mathbb{CP}^1\times M, \omega_{\mathrm{FS}}\oplus\omega)$ such that:
	\begin{itemize}
\item $\mathcal{J}_{reg}$ is a comeagre, i.e., it is a countable intersection of open dense subsets of  $\mathcal{J}_c(\mathbb{CP}^1\times M, \omega_{\mathrm{FS}}\oplus\omega)$.
		\item For every $J\in \mathcal{J}_{\mathrm{reg}}$ and generic point $p\in \mathbb{CP}^1\times M$, there exists a $J$-holomorphic sphere $u:(\mathbb{CP}^1,i)\to (\mathbb{CP}^1\times M,J)$ that passes through $p$ and represents the homology class $[\mathbb{CP}^1\times \{\operatorname{pt}\}] \in H_2(\mathbb{CP}^1\times M,\mathbb{Z} )$.
	\end{itemize}
\end{lemma} 
\begin{proof}
By Theorem \ref{trans0} and Remark \ref{trans123}, there exists a subset $\mathcal{J}_{\mathrm{reg}}$ of $\mathcal{J}_c(\mathbb{CP}^1\times M, \omega_{\mathrm{FS}}\oplus\omega)$ such that:
	\begin{itemize}
\item $\mathcal{J}_{\mathrm{reg}}$ is a comeagre, i.e., it is a countable intersection of open dense subsets of  $\mathcal{J}_c(\mathbb{CP}^1\times M, \omega_{\mathrm{FS}}\oplus\omega)$.
		\item For every $J\in \mathcal{J}_{\mathrm{reg}}$, the moduli space $\widehat{\mathcal{M}}(J, [\mathbb{CP}^1\times\{\operatorname{pt}\}])\times_{\operatorname{Aut}(\mathbb{CP}^1,i)} \mathbb{CP}^1$ is a smooth manifold of dimension $2n$.
	\end{itemize}
Pick an $\omega$-compatible almost complex structure $J_M$ on $M$. By Lemma \ref{split0}, we have $i\oplus J_M \in \mathcal{J}_{\mathrm{reg}}$. By Theorem \ref{trans1}, there exists a smooth path $\{J_t\}_{t\in [0,1]}\subset \mathcal{J}_c(\mathbb{CP}^1\times M, \omega_{\mathrm{FS}}\oplus\omega)$ with $J_0=i\oplus J_M$ and $J_1=J$ such that the moduli space
\begin{equation*}
	\mathcal{M}(\{J_t\}_{t\in [0,1]}, [\mathbb{CP}^1\times\{\operatorname{pt}\}]):=\left\{(t,u):
	\begin{array}{l}
		t\in [0,1],\\
		u:(\mathbb{CP}^1,i)\to (\mathbb{CP}^1\times M,J_t),\\
		du\circ i=J_t\circ du ,\\
		u_*[\mathbb{CP}^1]=[\mathbb{CP}^1\times\{\operatorname{pt}\}].
	\end{array}
	\right\}\bigg/\sim
\end{equation*}
produces a smooth cobordism  between $\widehat{\mathcal{M}}(i\oplus J_M,[\mathbb{CP}^1\times\{\operatorname{pt}\}])\times_{\operatorname{Aut}(\mathbb{CP}^1,i)} \mathbb{CP}^1$ and  $\widehat{\mathcal{M}}(J, [\mathbb{CP}^1\times\{\operatorname{pt}\}])\times_{\operatorname{Aut}(\mathbb{CP}^1,i)} \mathbb{CP}^1$. Moreover, this cobordism is compact by Theorem \ref{compact1}. Moreover, we have a well-defined evaluation map 
	\[\operatorname{ev}_{\{J_t\}} :\widehat{\mathcal{M}}(\{J_t\}_{t\in [0,1]}, [\mathbb{CP}^1\times\{\operatorname{pt}\}])\times_{\operatorname{Aut}(\mathbb{CP}^1,i)} \mathbb{CP}^1\to \mathbb{CP}^1\times M\]
defined by $\operatorname{ev}_{\{J_t\}}([(t,u,z)])=\operatorname{ev}_{J_t}([(u,z)])$. The map $\operatorname{ev}_{\{J_t\}}$ is a smooth homotopy from $\operatorname{ev}_{J_0}=\operatorname{ev}_{i\oplus J_M}$ to $\operatorname{ev}_{J}$. From Lemma \ref{split0}, the mod $2$ mapping degree of $\operatorname{ev}_{i\oplus J_M}$ does not vanish, i.e.,
\[\operatorname{deg}(\operatorname{ev}_{i\oplus J_M})=1\, (\text{mod } 2).\]
By the homotopy-invariance of mapping degree, we have 
\[\operatorname{deg}(\operatorname{ev}_{J})=\operatorname{deg}(\operatorname{ev}_{i\oplus J_M})=1\, (\text{mod } 2).\]	
This means that for generic point $p\in \mathbb{CP}^1\times M$, $\operatorname{ev}^{-1}_{J}(p)$ is not empty. In other words, there exists a $J$-holomorphic sphere $u:(\mathbb{CP}^1,i)\to (\mathbb{CP}^1\times M,J)$ that passes through $p$ and represents the homology class $[\mathbb{CP}^1\times \{\operatorname{pt}\}] \in H_2(\mathbb{CP}^1\times M,\mathbb{Z} )$.
\end{proof}
\begin{proof}[Proof of Theorem \ref{ab}]
Given $J \in \mathcal{J}_c(\mathbb{CP}^1\times M, \omega_{\mathrm{FS}}\oplus\omega)$ and point $p\in \mathbb{CP}^1\times M$.
	By  Lemma \ref{split2}, one can choose a sequence $J_n\in \mathcal{J}_{\mathrm{reg}}$ that $C^{\infty}$-converges to $J$, a sequence $p_n\in \mathbb{CP}^1\times M$ converging to $p$, and elements $[(u_n,J_n,z_n)]\in \widehat{\mathcal{M}}(J_n, [\mathbb{CP}^1\times\{\operatorname{pt}\}])\times_{\operatorname{Aut}(\mathbb{CP}^1,i)} \mathbb{CP}^1$ such that $u_n$ passes through $p_n$ at $z_n$ for each $n$. By Theorem \ref{compact1}, a subsequence of the sequence $[(u_n,J_n,z_n)]$ converges to some $[(u,J,z)]\in \widehat{\mathcal{M}}(J,[\mathbb{CP}^1\times\{\operatorname{pt}\}])\times_{\operatorname{Aut}(\mathbb{CP}^1,i)} \mathbb{CP}^1$ such that $u(z)=p$.
\end{proof}
\AtEndDocument{
 \bibliographystyle{alpha}
\bibliography{extremallagrangian}
\Addresses}
\end{document}